\documentclass[a4paper]{amsart}
\usepackage[english]{babel}
\usepackage[utf8]{inputenc}
\usepackage[T1]{fontenc}
\usepackage{xcolor}
\usepackage[foot]{amsaddr}
\usepackage{lipsum}
\usepackage[all]{xy}
\usepackage{hyperref}
\usepackage[nameinlink, capitalize]{cleveref}
\usepackage{amsmath,amsthm,amsfonts,amssymb,calrsfs}
\usepackage{tikz-cd, tikz}
\usetikzlibrary{hobby,calc,intersections}
\usepackage{yfonts}
\usepackage{extpfeil}
\usepackage{mathabx}
\usepackage[nameinlink, capitalize]{cleveref}
\newcommand{\lref}{\labelcref}
\usepackage{stmaryrd}

\newtheorem{thm}{Theorem}[section]
\newtheorem{prop}[thm]{Proposition}
\newtheorem{corol}[thm]{Corollary}

\newtheorem{lem}[thm]{Lemma}
\newtheorem*{thm*}{Theorem}
\newtheorem*{prop*}{Proposition}
\theoremstyle{definition}
\newtheorem{defi}[thm]{Definition}
\newtheorem{ex}[thm]{Example}

\theoremstyle{remark}
\newtheorem{rmq}[thm]{Remark}

\usepackage{xspace}
\newcommand{\ie}{i.e.\xspace}

\newcommand{\C}{\mathbb{C}}
\newcommand{\R}{\mathbb{R}}
\newcommand{\PP}{\mathbb{P}}
\newcommand{\Q}{\mathbb{Q}}

\newcommand{\N}{\mathbb{N}}
\newcommand{\D}{\mathbb{D}}
\newcommand{\F}{\mathcal{F}}

\DeclareMathOperator{\id}{id}
\DeclareMathOperator{\Div}{Div}

\DeclareMathOperator{\lcm}{lcm}

\DeclareMathOperator{\Num}{Num}

\DeclareMathOperator{\Hol}{Hol}

\DeclareMathOperator{\Sur}{Sur}

\DeclareMathOperator{\mult}{mult}

\DeclareMathOperator{\Aut}{Aut}

\newcommand{\bigslant}[2]{{\raisebox{.2em}{$#1$}\left/\raisebox{-.2em}{$#2$}\right.}}

\title[Finiteness results for hyperbolic orbifold pairs]{Finiteness results for hyperbolic orbifold pairs}
\author{Laurine Weibel}
\address{Univ Brest, CNRS UMR 6205,	Laboratoire de Mathematiques de Bretagne Atlantique\\ F-29200 Brest, France}
\email{\href{laurine.weibel@univ-brest.fr}{laurine.weibel@univ-brest.fr}}

\begin{document}

\begin{abstract}
    Noguchi proved that the set of dominant maps from a fixed variety to a fixed hyperbolic variety is finite. We extend this result to the setting of orbifold pairs, as introduced by Campana, under suitable assumptions. Certain compactness properties also allow us to prove that the set of orbifold pointed maps and the orbifold automorphism group are finite.
\end{abstract}

\maketitle

\tableofcontents

\section{Introduction}

In \cite{de_franchis_teorema_1913}, De Franchis proved a finiteness theorem for compact Riemann surfaces. Namely, for $X$ and $Y$ smooth compact Riemann surfaces with $Y$ of genus at least $2$, the set of surjective holomorphic maps $X\to Y$ is finite. Riemann surfaces with genus at least two caracterize both hyperbolic curves and curves of general type. On the one hand, considering Riemann surfaces of genus at least $2$ as hyperbolic curves, the result of De Franchis was generalized in higher dimension to hyperbolic varieties by Noguchi \cite{noguchi_meromorphic_1992}. He proved that, for $X$ and $Y$ compact irreducible complex spaces, if $Y$ is hyperbolic, then the set of dominant meromorphic maps from $X$ to $Y$ is finite. On the other hand, considering Riemann surfaces of genus at least $2$ as curves of general type, the result of De Franchis was generalized in higher dimension to varieties of general type by Kobayashi–Ochiai \cite{kobayashi_meromorphic_1975}. They proved that, for $X$ and $Y$ compact irreducible complex spaces, if $Y$ is of general type, then the set of dominant meromorphic maps from $X$ to $Y$ is finite.

The main purpose of this paper is to generalize some finiteness results pertaining to hyperbolic orbifold pairs introduced by Campana \cite{campana_orbifolds_2004,campana_orbifoldes_2011} in his program to study special varieties. He introduced the notion of orbifold base of a fibration $f:X\to Y$ considering a $\Q$-divisor on $Y$ defined by the multiple fibers of $f$, by means of a suitable ramification formula. 

An orbifold $(X,\Delta)$ is a pair consisting of $X$, a complex variety, with a $\Q$-divisor $\Delta$ on $X$ for which the decomposition in irreducible components is of the form 
    \[\Delta=\sum_i \left(1-\frac{1}{m_i}\right) \Delta_i,\]
    where $m_i\in\N\cup \{\infty\}$. 
    
    An orbifold $(X,\Delta)$ is said to be \emph{compact} (resp. \emph{projective})  if $X$ is compact (resp. projective) and if $\lfloor \Delta\rfloor$, the union of all components with infinite multiplicities, is empty.
    
    An orbifold $(X,\Delta)$ is said to be \emph{smooth} if $X$ is smooth and $\lceil \Delta\rceil$, the support of $\Delta$, is a simple normal crossing (s.n.c.) divisor.

An orbifold morphism $f:(X,\Delta_X)\to (Y,\Delta_Y)$ is a holomorphic map that satisfies some ramification conditions along $(Y,\Delta_Y)$ (\cref{def-orb-map}).

Some generalisations of De Franchis theorem were obtained in the context of orbifolds curves by Corlette–Simpson \cite{corlette_classification_2008}, Delzant \cite{delzant_trees_2008}
and Campana \cite{campana_fibres_2005}\footnote{Corlette–Simpson  \cite{corlette_classification_2008}, and Delzant \cite{delzant_trees_2008} proved results in the classical sense whereas Campana \cite{campana_fibres_2005} proved results in the non classical one (Campana's sense). In the classical sense, the ramification condition is a divisibility condition whereas in Campana's sense, the ramification condition is an inequality.}. In particular, Camapana established the following result: for $X$ a Riemann surface and $(Y,\Delta)$ a hyperbolic orbifold curve, there is only a finite number of surjective orbifold maps $X\to (Y,\Delta)$.

More recently, Bartsch–Javanpeykar \cite{bartsch_kobayashi-ochiais_2024} proved a generalization of the finiteness theorem of Kobayashi–Ochiai \cite{kobayashi_meromorphic_1975} for dominant rational maps in the setting of Campana’s orbifold maps.  Namely, for $X$ a variety and $(Y,\Delta)$ an orbifold pair of general type, there is only a finite number of dominant orbifold morphisms $X\to (Y,\Delta)$.

In this paper, we  prove several new finiteness results pertaining to hyperbolic orbifold pairs, under suitable assumptions, generalizing the results of De Franchis and Noguchi.

\begin{thm}
 Let $(X,\Delta_X)$ be a smooth projective orbifold and $(Y,\Delta_Y)$ be a hyperbolic smooth projective orbifold. If $(Y,\Delta_Y)$ is uniformizable or if $K_Y+\Delta_Y$ is pseudo-effective, then  the set of surjective orbifold maps from $X$ on $(Y,\Delta_Y)$ $\Sur((X,\Delta_X),(Y,\Delta_Y))$ is finite.
\end{thm}

For example, this result can be applied to an orbifold $(Y,\Delta)$ whith nef or big canonical bundle, and so we obtain another proof of the result of Bartsch-Javanpeykar in the case of hyperbolic varieties of general type.

The proof of these theorems follows the general strategy of Noguchi. Therefore, one of the main points of the proof is a compactness result on the set of orbifold morphisms.

\begin{thm}
Let $(X,\Delta_X)$ be an orbifold pair with $\lfloor \Delta_X\rfloor=\emptyset$. Let $(Y,\Delta_Y)$ be a compact orbifold pair. If $(Y,\Delta_Y)$ is hyperbolic, then $\Hol((X,\Delta_X),(Y,\Delta_Y))$ is relatively compact in $\Hol(X,Y)$.
\end{thm}

If we consider more specifically the surjective maps, the set is compact. 
\begin{corol}
    Let $(X,\Delta_X)$ be an orbifold pair with $\lfloor \Delta_X\rfloor=\emptyset$. Let $(Y,\Delta_Y)$ be a compact orbifold pair. If $(Y,\Delta_Y)$ is hyperbolic, then $\Sur((X,\Delta_X),(Y,\Delta_Y))$ is compact.
\end{corol}

In addition to the previous theorem, these results of compactness allow to prove others finitness results for the set of orbifold automorphisms $\Aut(Y,\Delta)$ and for the set of orbifold pointed maps.

\begin{thm}
    Let $(Y,\Delta)$ be a smooth compact orbifold. If $(Y,\Delta)$ is hyperbolic, then $\Aut(Y,\Delta)$ is finite.
\end{thm}

\begin{thm}
    Let $(X,\Delta_X)$ be an orbifold pair with $\lfloor \Delta_X\rfloor=\emptyset$. Let $(Y,\Delta_Y)$ be a hyperbolic compact orbifold pair. Let $x$ be a point in $X$ and $y$ be a point in $Y\backslash \lceil\Delta_Y\rceil$.
Then the set of holomorphic pointed orbifold maps  $\Hol\left[((X,\Delta_X),x),\left((Y,\Delta_Y),y\right)\right]$ is finite.
\end{thm}

The second point of the proof is a rigidity result. 

\begin{prop}
 Let $(X,\Delta_X)$ be a projective orbifold. Let $(Y,\Delta_Y)$ be a smooth projective hyperbolic orbifold.
 If $(Y,\Delta)$ is uniformizable or if the orbifold canonical bundle $K_Y+\Delta$ is pseudo-effective, then the set of surjective holomorphic orbifold morphisms from $(X,\Delta_X)$ to $(Y,\Delta_Y)$ is zero-dimensional.
\end{prop}

The original proof of Noguchi is based on a theorem proved by Miyaoka–Mori in the projective setting without boundary (see \cite{miyaoka_numerical_1986}). This approach works in the case of uniformizable pairs, otherwise it is a tricky point of the proof. 

In this document, we always assume all varieties to be complex varieties irreducible, reduced and normal.

\subsubsection*{Acknowledgements.} I would like to thank my Ph.D. advisors, Erwan Rousseau and Benoît Claudon, for their invaluable guidance and support throughout this work. I am also sincerely grateful to Ariyan Javanpeykar for the enlightening discussions during my research stay at Radboud University in Nijmegen, which contributed to the development of this project.

\section{Orbifolds}
\subsection{First definitions}
Let $X$ be a normal complex projective variety. First, we recall some notions about orbifolds introduced  by Campana in \cite{campana_orbifolds_2004}. These notions are also recalled in \cite{rousseau_hyperbolicity_2010}.

\begin{defi}
    An \emph{orbifold} $(X,\Delta)$ is a pair consisting of $X$ with a Weil $\Q$-divisor $\Delta$ on $X$ for which the decomposition in irreducible components is of the form 
    \[\Delta=\sum_i \left(1-\frac{1}{m_i}\right) \Delta_i,\]
    where $m_i\in\N\cup \{\infty\}$.
    An orbifold $(X,\Delta)$ is said to be \emph{compact} if $X$ is compact and if $\lfloor \Delta\rfloor$, the union of all components with infinite multiplicities, is empty.
    
    An orbifold $(X,\Delta)$ is said to be \emph{projective} if $X$ is projective and if $\lfloor \Delta\rfloor$, the union of all components with infinite multiplicities, is empty.
    
    An orbifold $(X,\Delta)$ is said to be \emph{smooth} if $X$ is smooth and $\lceil \Delta\rceil$, the support of $\Delta$, is a simple normal crossing (s.n.c.) divisor. This means that every component of $\lceil \Delta \rceil$ is smooth, and around any point of $X$, the divisor $\lceil \Delta \rceil$ can be locally described by an equation of the form $x_1 \dots x_d = 0$, for some $d \leq \dim X$.
\end{defi}

Following \cite{campana_brody_2009}, let us recall the definition of  orbifold morphisms from the unit disk to an orbifold.

\begin{defi}
    Let $(X,\Delta)$ be an orbifold with $\Delta=\sum_i\left(1-\frac{1}{m_i}\right)Z_i$,
    $\D = \{z \in \C\, |\, |z|<1\}$ the unit disk and $h$ a holomorphic map from $\D$ to $X$.
        $h$ is an \emph{orbifold morphism from $\D$ to $X$} if $h(\D) \not\subset \lceil \Delta\rceil$ and $\mult_x(h^\ast Z_i)\geq m_i$ for all $i$ and $x\in \D$ with $h(x)\in \lceil Z_i\rceil$. If $m_i=\infty$, we require $h(\D)\cap Z_i = \emptyset$.
\end{defi}

Also, we can define orbifold morphisms between orbifolds.

\begin{defi}\label{def-orb-map}
    Let $(X,\Delta_X)$ and $(Y,\Delta_Y)$ be orbifolds and $f:X\to Y$ a holomorphic map. $f$ is an \emph{orbifold morphism} from $(X,\Delta_X)$ to $(Y,\Delta_Y)$ if
    \begin{enumerate}
        \item $f(X)\not\subset \lceil \Delta_Y\rceil$, where $\lceil \Delta_Y\rceil$ denotes the support of $\Delta_Y$.
        \item for every irreducible divisors $D\subset Y$ and $E \subset X$ such that  $\linebreak f^\ast(D)=t_{E,D} E +R$, with $R$ an effective divisor of $X$ not containing $E$, we have
        \[t_{E,D}\cdot m_{X}(E) \geq m_Y(D),\]
        where $m_X$ (resp. $m_Y$) denotes the orbifold multiplicity on $X$ (resp. $Y$).
    \end{enumerate}
\end{defi}

We denote by $\Hol((X,\Delta_X),(Y,\Delta_Y))$ the set of holomorphic orbifold maps from $(X,\Delta_X)$ to $(Y,\Delta_Y)$.

\subsection{Orbifold bundles}

In this subsection we recall some statements regarding bundles on an orbifold. 
For the logarithmic cotangent bundle, we refer to \cite{Noguchi_logarithmic_1986}.

\subsubsection{Adapted coverings}

This subsection is based on \cite{claudon_expose_2015}.

\begin{defi}
    Let $(X,\Delta)$ be a smooth orbifold pair, with $\Delta=\sum \left(1-\frac{1}{m_i}\right)\Delta_i$. A \emph{$\Delta$-adapted covering} is a Galois ramified covering $\pi:Y\to (X,\Delta)$ such that
    \begin{enumerate}
        \item $Y$ is a smooth projective variety,
        \item $\pi$ ramifies exactly with multiplicity $m_i$ over $\Delta$,
        \item the support of $\pi^\ast \Delta+\mathrm{Ram}(\pi)$ has only normal crossings, and the support of the branch locus of $\pi$ has only normal crossings.
    \end{enumerate}
\end{defi}

\begin{rmq}
By \cite[Proposition~4.1.12]{lazarsfeld_positivity_2004-1}, there always exists such a $\Delta$-adapted covering.
\end{rmq}

We can describe locally this covering as follows.

Let $\pi:Y\to X$ be a $\Delta$-adapted covering. Let $y\in Y$ be a point. One can choose $(w_1,\dots, w_n)$ local coordinates centered in $y\in Y$ and $(z_1,\dots, z_n)$ local coordinates on $X$ centered in $\pi(y)$ such that $\pi$ is described by:
\[\pi(w_1,\dots, x_n)=\left(w_1^{m_1}, \dots, w_k^{m_k}, w_{k+1},\dots, w_{n-j}, w_{n-j+1}^{p_j},\dots, w_{n}^{p_1}\right),\]
assuming that $\pi\left(\{w_i=0\}\right)\subset \Delta$ for $i\in \ldbrack 1,k \rdbrack$ and $\pi\left(\{w_i=0\}\right)\subset \mathrm{Ram}(\pi)\backslash \pi^{-1}(\Delta)$ for $i\in \ldbrack n-j+1,n \rdbrack$.

\subsubsection{Orbifold (co)tangent bundle}

This subsection is also based on \cite{claudon_expose_2015}.

Let $(X,\Delta)$ be an orbifold with $\Delta=\sum\left(1-\frac{1}{m_i}\right)\Delta$.
We want to define the orbifold cotangent bundle. The main goal is to give a sense to $\frac{d z_i}{z_i^{1-\frac{1}{m_i}}}$. Let us consider $\pi:Y\to X$ a $\Delta$-adapted covering. 

\begin{defi}
    The \emph{orbifold cotangent bundle} associate to $\pi$, denoted by $\Omega^1(\pi,\Delta)$, is defined to be the subsheaf of $\pi^\ast \Omega^1_X(\log \lceil \Delta\rceil)$ fitting in the following short exact sequence:
\[0 \to \Omega^1(\pi,\Delta) \to \pi^\ast \Omega^1_X(\log\lceil \Delta\rceil)\overset{\pi^\ast \mathrm{res}}{\longrightarrow} \underset{m_i<\infty}{\bigoplus_{i\in I}} \mathcal{O}_{\pi^\ast \Delta_i/m_i} \to 0.\]
\end{defi}

We can describe the orbifold cotangent bundle locally. In coordinates as above, it is generated by the formal pull-back:
\[\left(w_1^{1-{m_1}}\pi^\ast d z_1,\dots, w_k^{1-{m_k}}\pi^\ast d z_k, \pi^\ast d z_{k+1},\dots, \pi^\ast d z_{n-j}, \pi^\ast d z_{n-j+1},\dots, \pi^\ast d z_n\right).\]

\begin{rmq}
    In local coordinates, it can also be rewritten as
    \[\left( d w_1,\dots,  d w_k,  d w_{k+1},\dots,  d w_{n-j}, w_{n-j+1}^{p_j-1} d w_{j+1},\dots, w_n^{p_1-1} d w_n\right).\]
\end{rmq}

In particular, one can note that $\Omega^1(\pi,\Delta)$ is locally free, that is why it is called orbifold cotangent \emph{bundle}.
One can note that the orbifold cotangent bundle is generated by the formal pull-back images of the forms $\frac{ d z_i}{z_i^{1-\frac{1}{m_i}}}$.

Once the notion of orbifold cotangent bundle is well defined, it is natural to consider its exterior powers. In particular, its determinant is defined by 
\[\det \left(\Omega^1(\pi,\Delta)\right)=\Omega^n(\pi, \Delta)=\pi^\ast \left(K_X+\Delta\right).\]

\begin{defi}\label{def-tang-orbi}
    The \emph{orbifold tangent bundle}, associated to $\pi$, denoted by $T(\pi,\Delta)$ is defined to be the dual of $\Omega_{(\pi,\Delta)}$.
\end{defi}

We can describe the orbifold tangent bundle locally. In coordinates as above, it is generated by the formal pull-back
\[\left(w_1^{m_1-1}\pi^\ast\frac{\partial}{\partial z_1},\dots, w_k^{m_k-1}\pi^\ast\frac{\partial}{\partial z_k}, \pi^\ast \frac{\partial}{\partial z_{k+1}} ,\dots, \pi^\ast \frac{\partial}{\partial  z_{n-j}}, \pi^\ast\frac{\partial}{\partial z_{n-j+1}},\dots, \pi^\ast\frac{\partial}{\partial z_{n}} \right).\]

\begin{rmq}\label{rmq-tangent-orbi}
   When $\pi:Y\to (X,\Delta)$ is orbi-étale, \ie $\mathrm{Ram}(\pi)\backslash \pi^{-1}(\Delta)=\emptyset$, the orbifold tangent bundle $T(\pi,\Delta)$ is exactly given by $TY$.
\end{rmq}

\subsubsection{Positivity of orbifold cotangent bundle}

In this subsection, we would like to extend the classical positivity notions for the orbifold cotangent bundle. This subsection is based on \cite{darondeau_quasi-positive_2020}.

\begin{defi}
    Given an adapted covering $\pi : Y\to (X,\Delta)$, the \emph{sheaf of orbifold symmetric differential
forms} is the direct image sheaf
\[S^{[N]}\Omega(X,\Delta) = \pi_\ast\left( \left(S^N\Omega(\pi, \Delta)\right)^{\Aut(\pi)}\right) \subset S^N\Omega_Y(\log \lceil\Delta\rceil).\]
Concretely, in local coordinates as above, it is generated by the 
\[\prod_{i=1}^n z_i^{\lceil\alpha_i/m_i\rceil}\left(\frac{dz_i}{z_i}\right)^{\alpha_i}\]
such that 
$\sum_{i=1}^n \alpha_i=N$.
\end{defi}

In view of the local expressione, we see that this sheaf is indeed independent of the choice
of $\pi$. Note that in general, the inclusion $S^{[qr]}\Omega_{(Y,\Delta)} \subset S^q\left(S^{[r]}\Omega_{(Y,\Delta)}\right)$ is strict. In fact, we
have equality if and only if $\Delta$ is purely logarithmic.

\begin{defi}
The orbifold pair $(X, \Delta)$ has a \emph{big cotangent bundle} if $\Omega_{(\pi, \Delta)}$ is big for some (hence for
all) adapted cover $\pi$. Equivalently, the orbifold cotangent bundle of $(X, \Delta)$ is big if for
some/any ample line bundle $A$ on $X$, there exists an integer $N$ such that
\[H^0\left(X, S^{[N]}\Omega(X,\Delta) \otimes A^{-1}\right) \neq \{0\}.\]
\end{defi}

Recall that the base locus of a vector bundle $E$ is defined by
\[\mathrm{Bs}(E)=\left\{x\in X\,|\, H^0(X,E)\to E_X\text{ is not surjective}\right\}.\]
To deal with ampleness, we use augmented base loci. 

\begin{defi}
The \emph{augmented base locus of $\Omega(X,\Delta)$}
is defined by
\[\mathbb{B}_+(\Omega(X,\Delta)) =
\bigcap_{N\geq 1}\bigcap_{\frac{p}{q}\in \Q}
\mathrm{Bs}\left(S^{[Nq]}\Omega(X,\Delta)\otimes A^{-Np}\right)\]
for some ample line bundle $A$ over $X$.
\end{defi}

\begin{rmq}
    Away from $\lceil\Delta\rceil$, this set turns out to be independent of the covering $\pi$.
\end{rmq}

There are many situations where we cannot expect global ampleness of $\Omega_{(\pi,\Delta)}$, that is why we recall the following intermediate positivity property introduced in \cite{darondeau_quasi-positive_2020}.

\begin{defi}
    We say that $(X,\Delta)$ has an \emph{ample cotangent bundle modulo boundary} if its orbifold augmented base locus is contained in the boundary.
\end{defi}

\begin{rmq}
    Equivalently, $(X,\Delta)$ has an ample cotangent bundle modulo boundary if for some (hence for all) adapted cover $\pi$, the orbifold cotangent bundle $\Omega(X,\Delta)$ is ample modulo $\Aut(\pi)$-invariant closed subset living over the boundary.

    Note that notion of ampleness modulo boundary of $\Omega_{(\pi,\Delta)}$ does not depend on $\pi$ (see \cite[Proposition 2.3]{darondeau_quasi-positive_2020}).
\end{rmq}

\begin{rmq}\label{ample-cover-ample-mod-delta}
   If $\pi:Y\to (X,\Delta)$ is a $\Delta$-étale covering and if $\Omega_Y$ is ample, then $(X,\Delta)$ has an ample cotangent bundle modulo $\lceil\Delta\rceil$. Indeed, since $\Omega_Y$ is ample, $\mathbb{B}_+(\Omega_Y)=\emptyset$ and so $\mathbb{B}_+\left(\Omega_{(X,\Delta)}\right)$ is empty over $X\backslash\lceil \Delta\rceil$.
\end{rmq}

We also recall an orbifold version of the fundamental vanishing theorem (see \cite[Proposition 5.6]{darondeau_quasi-positive_2020}).

 \begin{thm}[Orbifold fundamental vanishing theorem]\label{vanishing-thm-orbi}
    Let $(X,\Delta)$ be a smooth orbifold pair. Then any orbifold entire curve is contained in $\mathbb{B}_+\left(\Omega(X,\Delta)\right)$, where $\mathbb{B}_+\left(\Omega(X,\Delta)\right)$ denotes the augmented base locus.
\end{thm}

As a direct consequence, we have the following proposition.  
\begin{prop}\label{ample-modulo-implies-hyp}
Let $(X,\Delta)$ be a smooth projective orbifold. If $(X,\Delta)$ has an ample cotangent bundle modulo $\lceil\Delta\rceil$, then $(X,\Delta)$ does not admit any non-constant orbifold entire curve $\C \to (X,\Delta)$, \ie $(X,\Delta)$ is said to be
\emph{Brody-hyperbolic}.
\end{prop}

\section{Orbifold hyperbolicity}

\subsection{Definitions and first properties}

\begin{defi} Let $(X,\Delta)$ be an orbifold.
        The \emph{orbifold Kobayashi pseudo-distance} $d_{(X,\Delta)}$ on $(X,\Delta)$ is the largest pseudo-distance on $X\backslash \lfloor \Delta\rfloor$ such that 
        \[g^\ast  d_{(X,\Delta)} \leq d_P\]
        for every orbifold morphism $g:\D\to (X,\Delta)$, where $d_P$ denotes the Poincaré distance on $\D$.
\end{defi}

As an immediate consequence of the definition, we have

\begin{prop}\label{distance-decreasing-orbi}
    Let $f:(X,\Delta_X) \to (Y,\Delta_Y)$ be an orbifold morphism. Then $f$ is distance-decreasing, \ie
    \[f^\ast d_{(Y,\Delta_Y)} \leq d_{(X,\Delta_X)} 
    .\]
\end{prop}

\begin{defi}
 An orbifold $(X,\Delta)$ is \emph{hyperbolic}
if the orbifold Kobayashi pseudo-distance 
$d_{(X,\Delta)}$ is a distance on
$X\backslash \lfloor \Delta\rfloor$.
\end{defi}

\subsection{One dimensional case}

Due to the Uniformization Theorem for compact Riemann surfaces, it is well known that a compact Riemann surface $X$ is hyperbolic if, and only if, the genus of $X$ satisfies $g_X\geq 2$. 

There is a similar characterization for the orbifold curves considering the degree of the orbifold canonical bundle.

\begin{prop}[See Corollary 5 in \cite{campana_brody_2009}]\label{dim1-orbi}
    Let $(X,\Delta)$ be a smooth compact orbifold curve.
Then $(X,\Delta)$ is hyperbolic if and only if $K_X+\Delta>0$.

More precisely, $(X,\Delta)$ is not hyperbolic if and only if  one of the following conditions hold:
\begin{enumerate}
    \item $X$ is an elliptic curve and $\Delta$ is empty.
    \item $X \simeq \mathbb{P}^1$ and $\lceil \Delta \rceil$ contains at most two points.
    \item $X \simeq \mathbb{P}^1$ and there are numbers: $p \leq q \leq r \in \N \cup \{\infty\}\backslash\{1\}$ such that $(X,\Delta)$
is isomorphic to
\[\left(\mathbb{P}^1,\left(1-\frac{1}{p}\right)\{0\}+\left(1-\frac{1}{q}\right)\{1\}+\left(1-\frac{1}{r}\right)\{\infty\}\right)\]
and $\frac{1}{p}+\frac{1}{q}+\frac{1}{r}\geq 1$ (there are exactly 5 possibilities for $(p,q,r)$: $(2,3,4)$, $(2,3,5),$ $(2,3,6)$,\, $(2,4,4)$,\, $(3,3,3)$).
    \item There is a point $\lambda\in \C\backslash\{0,1\}$ such that $(X,\Delta)$ is isomorphic to
    \[\left(\mathbb{P}^1,\left(1-\frac{1}{2}\right)\{0\}+\left(1-\frac{1}{2}\right)\{1\}+\left(1-\frac{1}{2}\right)\{\infty\}+\left(1-\frac{1}{2}\right)\{\lambda\}\right).\]
\end{enumerate}
\end{prop}    

\subsection{\texorpdfstring{Hyperbolic lines arrangment in $\mathbb{P}^2$}{Hyperbolic line arrangment in P2}}

The following example is one of Deligne-Mostow examples (see \cite[Section 4.11]{holzapfel_ball_1998}).

Let $P_1,\, P_2,\, P_3,\, P_4$ be four points in general position on $\PP^2$. The projective line through $P_i,\, P_j,\, i\neq j$ is denoted by $L_{ij} = L_{ji}$.
The six lines $L_{ij}$ are drawn in \cref{quadrihedral}.

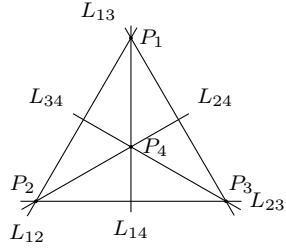
\begin{figure}[h!]
    \centering
    \begin{tikzpicture}[scale=0.5]
    \begin{footnotesize}
\draw (7.030923057358191,3.31408933472715)-- (11.432812864648628,5.855521599909354) node[above right] {$L_{24}$};
\draw (9.908582755888318,8.188438280660327)-- (9.908582755888316,3.2423912736190745) node[below] {$L_{14}$};
\draw (12.796175688186956,3.308354379536975) -- (8.383222213533088,5.85617425604968)node[above left] {$L_{34}$};
\draw (12.625547509863093,3.1563372859995953) -- (9.752464606495895,8.132662848989092)node[above left] {$L_{13}$};
\draw (7.001750562457261,3.5321312633355912) -- (12.816067160374057,3.5321312633355912) node[right] {$L_{23}$};
\draw (10.067408811903627,8.13735308084208) -- (7.18161548426111,3.139012417221984) node[below] {$L_{12}$};
\filldraw (7.408582755888316,3.5321312633355912) circle (1pt);
\draw (7.608582755888316,3.5321312633355912) node[above left] {$P_2$};
\filldraw (12.408582755888316,3.5321312633355912) circle (1pt);
\draw (12.808582755888316,3.5321312633355912) node[above] {$P_3$};
\filldraw (9.908582755888318,4.975506936309656) circle (1pt);
\draw (10, 4.975506936309656) node[right] {$P_4$};
\filldraw (9.90858275588832,7.862258282257785) circle (1pt) node[right] {$P_1$};
\end{footnotesize}
\end{tikzpicture}
    \caption{Quadrihedral}
    \label{quadrihedral}
\end{figure}

 We blow up the four triple points $P_i$. The blown up surface is denoted by $Y$. As in \cref{blowup}, we denote the exceptional lines
on $Y$ by $L_{0j}, \, j \in \ldbrack 1,4\rdbrack$.

\begin{figure}[h!]
    \centering
     \begin{tikzpicture}[scale=0.5]
    \begin{footnotesize}
\draw (0.6822866862816246,1)-- (4.079798939022026,1) node[right] {$L_{23}$};
\draw (3.7742479663158934,0.8120921012754374)-- (5.482645244232453,3.7711229861392854) node[right] {$L_{03}$};
\draw (5.490771129832862,3.4109549898509446)-- (3.759030506108087,6.410417735673274) node[above] {$L_{13}$};
\draw  (4.087573566602667,6.196152422706633)-- (0.690061313862266,6.196152422706634) node[left] {$L_{01}$};
\draw (1.006116664854312,6.409852890535482)-- (-0.7319777326345875,3.399385085733896) node[left] {$L_{12}$};
\draw (-0.7328261921688676,3.798236911994281)-- (1.0016612921654329,0.7940164640349434) node[below] {$L_{02}$};
\draw (2.382736642227377,3.598076211353316) circle (1cm);
\draw (3.5,3.598076211353316) node[right] {$L_{04}$};
\draw (1.3435061576860414,4.198076211353322) -- (-0.01136664701203266,4.980312389863459) node[left] {$L_{34}$};
\draw (1.6899163191998368,3.99807621135331)-- (2.209531561470482,3.6980762113533205);
\draw (2.555941722984273,3.498076211353311)-- (4.799154752498138,2.2029565649138045);
\draw (2.382736642227377,6.402730366257467)-- (2.382736642227377,3.7980762113533193);
\draw (2.382736642227377,3.3980762113533123)-- (2.382736642227377,2.798076211353315);
\draw (2.3827366422273766,2.398076211353318)-- (2.3827366422273766,0.80577297398705) node[below] {$L_{14}$};
\draw (3.421967126768692,4.19807621135331) -- (4.810179551011993,4.999561028182562) node[right] {$L_{24}$};
\draw (3.075556965254937,3.9980762113533226)-- (2.555941722984268,3.698076211353319);
\draw (2.2095315614704854,3.4980762113533155)-- (-0.0018505508370246226,2.2213341535314512);
    \end{footnotesize}
\end{tikzpicture}
    \caption{Blown up surface}
    \label{blowup}
\end{figure}
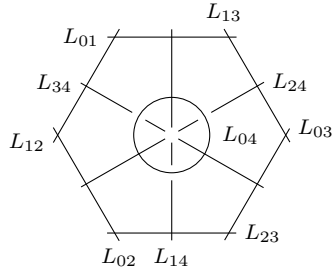

We denote by $\nu_{ij}$ the weight of $L_{ij}$ and we set
\[\Delta=\sum_{0\leq i<j\leq 4} \left(1-\frac{1}{\nu_{ij}}\right) L_{ij}.\]

In \cref{weights-table} we arrange the weights $\nu_{ij}$ in a table putting the $\nu_{0j}$'s into the middle part.

\begin{figure}[h!]
    \centering
    \begin{tabular}{|cccc|}
\cline{1-4}
$\nu_{14}$                       & $\nu_{13}$ & $\nu_{24}$                     &  $\nu_{23}$\\ \cline{2-3}
\multicolumn{1}{|c|}{$\nu_{12}$} & $\nu_{01}$ & \multicolumn{1}{c|}{$\nu_{02}$}& $\nu_{34}$\\
\multicolumn{1}{|c|}{$\nu_{34}$} & $\nu_{03}$ & \multicolumn{1}{c|}{$\nu_{04}$}& $\nu_{12}$\\ \cline{2-3}
$\nu_{14}$                       & $\nu_{24}$ & $\nu_{13}$                      & $\nu_{23}$\\ \cline{1-4}
\end{tabular}
    \caption{Table of weights}
    \label{weights-table}
\end{figure}

\begin{prop}\label{Deligne-Mostow-hyp}
    If the weights are one of the weights listed in \cref{hyperbolic-weights}, then the orbifold pair $(Y,\Delta)$ is (Kobayashi)-hyperbolic.
\end{prop}

\begin{figure}[h!]
    \centering
\begin{tabular}{ccccccc}
  \begin{tabular}{|cccc|}
\cline{1-4}
9                       & 9 & 9                      & 9\\ \cline{2-3}
\multicolumn{1}{|c|}{9} & 3 & \multicolumn{1}{c|}{3} & 9\\
\multicolumn{1}{|c|}{9} & 3 & \multicolumn{1}{c|}{3} & 9\\ \cline{2-3}
9                       & 9 & 9                      & 9\\ \cline{1-4}
\end{tabular} &
&
\begin{tabular}{|cccc|}
\cline{1-4}
6                       & 12& 6                      & 12\\ \cline{2-3}
\multicolumn{1}{|c|}{12}& 3 & \multicolumn{1}{c|}{3} & 6 \\
\multicolumn{1}{|c|}{6} & 3 & \multicolumn{1}{c|}{4} & 12\\ \cline{2-3}
6                       & 6 & 12                     & 12\\ \cline{1-4}
\end{tabular}&
&
  \begin{tabular}{|cccc|}
\cline{1-4}
5                       & 15 & 5                      & 15\\ \cline{2-3}
\multicolumn{1}{|c|}{15}& 3  & \multicolumn{1}{c|}{3} & 5 \\
\multicolumn{1}{|c|}{5} & 3  & \multicolumn{1}{c|}{5} & 15\\ \cline{2-3}
5                       & 5  & 15                     & 15\\ \cline{1-4}
\end{tabular} &
&
\begin{tabular}{|cccc|}
\cline{1-4}
6                       & 6 & 6                      & 6 \\ \cline{2-3}
\multicolumn{1}{|c|}{6} & 4 & \multicolumn{1}{c|}{4} & 6 \\
\multicolumn{1}{|c|}{6} & 4 & \multicolumn{1}{c|}{4} & 6 \\ \cline{2-3}
6                       & 6 & 6                      & 6 \\ \cline{1-4}
\end{tabular}\\
&&&&&&\\
\end{tabular}
    \begin{tabular}{ccccc}
  \begin{tabular}{|cccc|}
\cline{1-4}
4                       & 8 & 4                      & 8 \\ \cline{2-3}
\multicolumn{1}{|c|}{8} & 4 & \multicolumn{1}{c|}{4} & 4 \\
\multicolumn{1}{|c|}{4} & 4 & \multicolumn{1}{c|}{8} & 8 \\ \cline{2-3}
4                       & 4 & 8                      & 8 \\ \cline{1-4}
\end{tabular} &
&
\begin{tabular}{|cccc|}
\cline{1-4}
5                       & 5 & 5                      & 5 \\ \cline{2-3}
\multicolumn{1}{|c|}{5} & 5 & \multicolumn{1}{c|}{5} & 5 \\
\multicolumn{1}{|c|}{5} & 5 & \multicolumn{1}{c|}{5} & 5 \\ \cline{2-3}
5                       & 5 & 5                      & 5 \\ \cline{1-4}
\end{tabular}&
&
\begin{tabular}{|cccc|}
\cline{1-4}
4                       & 4 & 4                      & 4 \\ \cline{2-3}
\multicolumn{1}{|c|}{6} & 6 & \multicolumn{1}{c|}{6} & 3 \\
\multicolumn{1}{|c|}{3} & 12& \multicolumn{1}{c|}{12}& 6 \\ \cline{2-3}
4                       & 4 & 4                      & 4 \\ \cline{1-4}
\end{tabular}
\end{tabular}
    \caption{Hyperbolic weights}
    \label{hyperbolic-weights}
\end{figure}

Before proving this result, we recall the following criterion that characterize (Kobayashi-)hyperbolicity considering Brody hyperbolicity for induced orbifold structures (see \cite[Theorem 5.5]{darondeau_quasi-positive_2020}).

\begin{thm}[Orbifold Brody’s criterion]\label{Brody-orbi}
Consider a smooth orbifold pair $(X,\Delta)$ with $\Delta=\sum\left(1-\frac{1}{m_i}\right)\Delta_i$. 
For a subset $I$ of $\{0,\dots,d\}$,
let $X_I= \bigcap_{i\in I} \Delta_i$,
and let $\Delta_{\overline{I}} =\sum_{j\notin I} \left(1-\frac{1}{m_j}\right)\Delta_j\cap X_I$. If all pairs
$(X_I , \Delta_{\overline{I}})$ are Brody-hyperbolic, then the pair $(X, \Delta)$ is (Kobayashi-)hyperbolic.
\end{thm}

\begin{rmq}
    We recall that an orbifold pair $(X,\Delta)$ is said to be \emph{Brody-hyperbolic} if it does not admit any non-constant orbifold entire curve $\C \to (X,\Delta)$.
\end{rmq}

\begin{proof}[Proof of \cref{Deligne-Mostow-hyp}]
To show that the orbifold pair is hyperbolic, we apply \cref{Brody-orbi}. 

The first step is to prove that $(Y,\Delta)$ is Brody-hyperbolic.
By \cite[Section 4.11]{holzapfel_ball_1998}, each case of \cref{hyperbolic-weights} corresponds to compact ball quotient
surfaces. Hence there exists a smooth quotient $X=\mathbb{B}/\Gamma$, where $\mathbb{B}$ denotes the unit ball, such that $X\to (Y,\Delta)$ is $\Delta$-étale. Then the cotangent bundle $\Omega_X$ is ample and so  $(Y,\Delta)$ has an ample cotangent bundle modulo $\Delta$ (see \cref{ample-cover-ample-mod-delta}). 
And so by \cref{ample-modulo-implies-hyp}, $(Y,\Delta)$ is Brody-hyperbolic.

The second step is to prove that each line $L_{ij}$ endowed with the induced orbifold structure is again Brody-hyperbolic. We are back to the study of orbifold curves, so it suffies to look at the degree of the orbifold canonical bundle (see \cref{dim1-orbi}).
By an easy computation, one can show that computing the degree of orbifold canonical bundle of $L_{ij}$ endowed with the induced orbifold structure returns to compute $K_{ij}=-2+ \sum_{k,l} \left(1-\frac{1}{\nu_{kl}}\right)$ for all pairs $(k,l)$ such that $L_{kl}$ intersects $L_{ij}$, and so one can see that $K_{ij}>0$. 

This conclude the proof that $(Y,\Delta)$ is hyperbolic for each weights listed in \cref{weights-table}.
\end{proof}

\section{Compactness result on the subset of surjective orbifold morphisms}

In this section, we prove a relative compactness result on $\Hol((X,\Delta_X),(Y,\Delta_Y))$, the set of holomorphic orbifold morphisms between two orbifold pairs $(X,\Delta_X)$ and $(Y,\Delta_Y)$, and more specifically a compactness result on $\Sur((X,\Delta_X),(Y,\Delta_Y))$, the set of surjective holomorphic orbifold morphisms from $(X,\Delta_X)$ to $(Y,\Delta_Y)$.

\begin{thm}\label{Compactness-hol}
Let $(X,\Delta_X)$ be an orbifold pair with $\lfloor\Delta_X\rfloor=\emptyset$. Let $(Y,\Delta_Y)$ be a compact orbifold. If $(Y,\Delta_Y)$ is hyperbolic, then $\Hol\left(\left(X,\Delta_X\right),\left(Y,\Delta_Y\right)\right)$ is relatively compact in $\Hol(X,Y)$.
\end{thm}

\begin{proof}
    By assumption, $(Y,\Delta_Y)$ is hyperbolic. Then the orbifold Kobayashi pseudo-distance $ d_{(Y,\Delta_Y)}$ is a distance on $Y \backslash \lfloor \Delta_Y\rfloor$. Since $(Y,\Delta_Y)$ is compact, there is no component with infinite multiplicities, \ie $\lfloor\Delta_Y\rfloor=\emptyset$. So that $Y$ is a compact metric space. By \cref{distance-decreasing-orbi}, each map $f\in \Hol((X,\Delta_X),(Y,\Delta_Y))$ is distance-decreasing \ie
    \[\forall (x,y)\in X^2 \quad d_{(Y,\Delta_Y)} (f(x),f(y))\leq d_{(X,\Delta_X)}(x,y).\]
    For all $x\in X$ and all $\varepsilon>0$, we define the set \[\{y\in X \,|\, d_{X}(x,y)<\varepsilon\}= d_{(X,\Delta_X),x}^{-1}\left([0,\varepsilon[\right).\]
    Since this set is an inverse image of an open set by a continuous map on $X\backslash \lfloor\Delta_X\rfloor=X$ (see \cite{campana_brody_2009}), this subset is a neighborhood of $x\in X$.
    Then, for a fixed $x\in X$ and $y\in d_{(X,\Delta_X),x}^{-1}([0,\varepsilon[)$, 
    \[d_{(Y,\Delta_Y)} (f(x),f(y))\leq d_{(X,\Delta_X)}(x,y)<\varepsilon.\]
    So, $\Hol((X,\Delta_X),(Y,\Delta_Y))$ is equicontinous at every point $x \in X$.

Since $\{f(x)\in Y\, | \, f \in \F \}\subset Y$, with $Y$ a compact variety, this subspace is relatively compact in $Y$.

By the Arzela-Ascoli Theorem (see \cite[Theorem~(1.3.1)]{kobayashi_hyperbolic_1998}), $\Hol((X,\Delta_X),(Y,\Delta_Y))$ is relatively compact in $\mathcal{C}(X,Y)$. 
Since $\Hol((X,\Delta_X),(Y,\Delta_Y))\subset \Hol(X,Y)$ and since the latter is closed inside $\mathcal{C}(X,Y)$, $\Hol((X,\Delta_X),(Y,\Delta_Y))$ is relatively compact in $\Hol(X,Y)$.
\end{proof}

We recall the result proved by Campana-Winkelmann in \cite{campana_brody_2009}, which establishes the closedness property of $\Hol((X, \Delta_X), (Y, \Delta_Y))$.

\begin{prop}[See Proposition~7 in \cite{campana_brody_2009}]\label{morph-orbi-ferme}
Let $(X,\Delta_X)$ and $(Y,\Delta_Y)$ be two orbifolds. Let $f_n:(X, \Delta_X)
\to (Y, \Delta_Y)$
be a sequence of orbifold morphisms. Assume that $(f_n)$, regarded as a sequence of holomorphic
maps from $X$ to $Y$, converges locally uniformly to a holomorphic map
$f: X \to Y$.

Then either $f(X) \subset \lceil \Delta_Y\rceil$ or $f$ is an orbifold morphism 
from $(X,\Delta_X)$ to $(Y,\Delta_Y)$.
\end{prop}

This result allows us to prove the following compactness result for the set of surjective orbifold morphisms $\Sur((X,\Delta_X),(Y,\Delta_Y))$.

\begin{corol}\label{Sur-compact}
Let $(X,\Delta_X)$ be an orbifold pair with $\lfloor\Delta_X\rfloor=\emptyset$. Let $(Y,\Delta_Y)$ be a compact orbifold.  If $(Y,\Delta_Y)$ is hyperbolic, then $\Sur((X,\Delta_X),(Y,\Delta_Y))$ is compact.
\end{corol}

\begin{proof}
    By \cref{Compactness-hol}, $\Sur((X,\Delta_X),(Y,\Delta_Y))$ is relatively compact in $\Hol(X,Y)$. 

To conclude, it remains to show that $\Sur((X,\Delta_X),(Y,\Delta_Y))$ is closed in $\Hol(X,Y)$.

Let $(f_n)$ be a sequence of surjective orbifold morphisms from $(X,\Delta_X)$ to $(Y,\Delta_Y)$ which converges to $f\in \Hol(X,Y)$. Since the set of surjective morphisms is closed (see \cite[Corollary~(5.3.5)]{kobayashi_hyperbolic_1998}), $f$ is also surjective. Thus, $f(X)$ cannot be contained in one of the components of $\Delta_Y$. By \cref{morph-orbi-ferme}, we conclude that $f\in \Sur((X,\Delta_X),(Y,\Delta_Y))$, hence $\Sur((X,\Delta_X),(Y,\Delta_Y))$ is closed, which concludes the proof of the corollary.
\end{proof}

\section{Reduction of the study}

In this section, we will show how we can reduce the study to the case of finite orbifold maps between a variety $X$ and an orbifold pair $(Y,\Delta)$.

Let $(X,\Delta_X)$ and $(Y,\Delta_Y)$ be two orbifold pairs.
Let $\mathcal{F}$ be a connected family of orbifold morphisms $f:(X,\Delta_X)\to (Y,\Delta_Y)$.

First of all, we can consider an adapted cover $\pi_X:\widetilde{X}\to (X,\Delta_X)$, such adapted cover exist by \cite[Proposition~4.1.12]{lazarsfeld_positivity_2004-1}.
We are now reduced to study morphisms $\widetilde{f}:\widetilde{X}\to (Y,\Delta)$. We denote by $\widetilde{\F}$ the corresponding family.

Secondly, we apply the simultaneous Stein factorization to the family $\widetilde{\mathcal{F}}$. We recall this result, described in \cite{kobayashi_hyperbolic_1998} (see also \cite{horst_compact_1985}).

\begin{prop}[Corollary~(5.3.2)\cite{kobayashi_hyperbolic_1998}]
    Let $X$ and $Y$ be two connected complex spaces, and let $\F\subset \Hol(X,Y)$ be a subfamily with connected universal complex structure such that each $f\in \F$ is a proper map from X into Y. Then $\F$ admits a simultaneous factorization 
    \[f:X\overset{p}{\longrightarrow} X' \overset{f'}{\longrightarrow} Y, \quad f\in \F\]
    through a common complex space $X'$ so that $p$ is a proper surjective holomorphic map with connected fibers and $f'$ is a finite map. 
\end{prop}
In our case, we use the following notations:
\[\xymatrix{
    & X' \ar[rd]^{f'}& \\
    \widetilde{X}\ar[rr]_{\widetilde{f}}\ar[ru]^p&  & (Y,\Delta_Y)
  }
  \]
Now, the problem is that $f':X'\to Y$ may not be an orbifold map anymore. 

\begin{ex}
    Let $E$ be an elliptic curve and $C$ a hyperelliptic curve. Let $i:C\to C$ an involution, such that $\bigslant{C}{<i>}\simeq \PP^1$. Let $t\in E$ a point of order 2. Consider the quotient $X$,  defined by $E\times C$ where one identifies $(x,y)$ with $(x+t, i(y))$. The map $X \to \bigslant{C}{<i>}= \PP^1$ is an elliptic fibration with $2g+2$ multiple fibers. 
    Then the Stein factorization is given by the following diagram, where $\mathrm{Br}(i)$ denotes the branch locus of $i$.
    \[\xymatrix{
    & \PP^1 \ar[rd]^{\id}& \\
    X \ar[ru]^f\ar[rr]_f & &\left(\PP^1,\frac{1}{2}\mathrm{Br}(i)\right)
  }\]
  The identity map is clearly not an orbifold map.
\end{ex}

To solve this problem, we add an orbifold structure on $X'$. Consider the orbifold base of the fibration $p$ (see \cite[(part 1.2.2.)]{campana_orbifolds_2004}), which does not depend on the map $f:X\to (Y,\Delta)$ by construction of the simultaneous Stein factorization. 

\begin{lem}
    Let $f:X\to (Y,\Delta)$ be an orbifold map between $X$ a variety and $(Y,\Delta)$ an orbifold pair. We denote by  \[f:X\overset{p}{\longrightarrow} X' \overset{f'}{\longrightarrow} Y\] the Stein factorization of $f$ and we denote by $(X', \Delta_p)$ the orbifold base of the fibration $p$.
    Then $f':(X',\Delta_p)\to (Y,\Delta)$ is an orbifold map.
\[\xymatrix{
    & (X',\Delta_p) \ar[rd]^{f'}& \\
    X \ar[ru]^p\ar[rr]_f & &\left(Y,\Delta\right)
  }\]
\end{lem}

\begin{proof}
    First of all, $f'(X')$ is not in the support of $\Delta$ because $f'(X')=f(X)$, where $f$ is an orbifold map and so $f(X)\not\subset \lceil \Delta\rceil$.

The second step is to show that for every irreducible divisors $D\subset Y$ and $E \subset X'$ such that $f'^\ast(D)=t_{E,D} E +\widetilde{E}$, with $\widetilde{E}$ an effective divisor of $X'$ not containing $E$, we have
        $t_{E,D} m_{X'}(E) \geq m_Y(D)$
        where $m_{X'}$ (resp. $m_Y$) denotes the orbifold multiplicity on $X'$ (resp. $Y$).
By definition of the orbifold base, the multiplicity $m_{X'}(E)$ is the infimum of all the multiplicity $m_i$ in the following pull-back
$p^\ast E= \sum m_i D_i +R$, with $R$ a $p$-exceptional divisor. 

Since $f=f'\circ p$, then we can compute it as follows.
\begin{align*}
    f^\ast \Delta_i &= (f'\circ p)^\ast \Delta_i\\
    &= p^\ast(t_{E,D} E +\widetilde{E})\\
    &= t_{E,D} p^\ast(E) +p^\ast \widetilde{E}\\
    &= t_{E,D} \sum m_i D_i + \widetilde{R}.
\end{align*}
Since $f$ is an orbifold map, the inequality $t_{E,D}m_i \geq m_Y(\Delta_i)$ holds for each $i$. Furthermore $m_{X'}=\inf(m_i)$, so we conclude.
\end{proof}

Now, the study is reduced to the case of finite morphisms between orbifold pairs.

Finally, let us consider an adapted cover $\pi_{X'}:\widehat{X}\to (X',\Delta_p)$. By definition of $\widehat{X}$ and construction of the simultaneous Stein factorization, the composed map $\widehat{X}\to (Y,\Delta)$ is finite.

We shall therefore assume that we study finite orbifold morphisms between a variety $X$ and an orbifold pair $(Y,\Delta)$.

\begin{rmq}
    All results mentioned in the introduction are stated in their more general form, involving morphisms between orbifold pairs. However, in the following pages, we will present them in the context described in this section.
\end{rmq}

\section{The evaluation morphism is an orbifold morphism}

In this section, we prove that the evaluation morphism is an orbifold morphism. This result was already proved by Bartsch–Javanpeykar–Rousseau \cite[Theorem 3.5.]{bartsch_weakly-special_2023} in the algebraic context. We propose an analytic proof. 

Let $X$ be a variety and $(Y,\Delta)$ be an orbifold pair, with $\Delta=\sum_i \left(1-\frac{1}{m_i}\right) \Delta_i$.

Let $\F\subset \Hol(X,(Y,\Delta))$ be a normal locally closed subset of holomorphic orbifold maps.
We define the evaluation morphism $ev$ by
\[\begin{array}{c|ccc}
    ev: & X \times \F & \longrightarrow & Y \\
     & (x,f) & \longmapsto & f(x)
\end{array}\]
Let $x$ be a point in $X$. We define the morphism $ev_x$, the evaluation at $x$, by
\[\begin{array}{c|ccc}
    ev_x : & \F & \longrightarrow & Y \\
     & f & \longmapsto & f(x)
\end{array}\]

\begin{thm}\label{ev-orbi}
The evaluation morphism \[\begin{array}{c|ccc}
    ev : & X \times \F & \longrightarrow & (Y,\Delta) \\
     & (x,f) & \longmapsto & f(x)
\end{array}\]
is an orbifold morphism.
\end{thm}

\begin{proof}
First of all, let us show that $ev(X\times \F)\not\subset \lceil \Delta\rceil$.

When one has fixed the map $f\in \F$, we set \[ev(X,f)=\{f(X)\in Y\,|\, x \in X\}=f(X) \subset Y.\] 
Since each map $f \in \F$ is an orbifold morphism, then $f(X)\not\subset \lceil \Delta\rceil$.

So $ev(X,\F)=\bigcup_{f\in \F} ev(X,f) $ is not included in the support of $\Delta$.

    The second step is to show that for every irreducible divisors $D\subset Y$ and $E \subset X$ such that $ev^\ast(D)=t_{E,D} E +R$, with $R$ an effective divisor of $X$ not containing $E$, we have
        $t_{E,D} m_{X\times \F}(E) \geq m_Y(D)$
        where $m_{X\times\F}$ (resp. $m_Y$) denotes the orbifold multiplicity on $X\times \F$ (resp. $Y$).
Since there is no orbifold structure on $X\times \F$, it remains to show that 
\[ev^\ast (\Delta_i)= \sum_{j} n_{i,j} E_{i,j}, \text{ with } n_{i,j}\geq m_i,\]
where $\Delta_i\in\lceil \Delta\rceil$ and $m_i$ denotes the orbifold multiplicity of $\Delta_i$ in $\Delta$.

Let $\Delta_i \in \lceil \Delta\rceil$ and let $E_{i,j}$ be a component of $\lceil ev^\ast (\Delta_i)\rceil$, $E_{i,j}$ is a divisor on $X\times \F$. By simplicity, we denote $E_{i,j}$ by $E$ instead.

If there exists $D\subset \F$, a divisor on $\F$, such that $E=X\times D$, then for all $f\in D$ $f(X) \subset \Delta_i$ by construction. This situation is not allowed since $\F$ is a subset of the orbifold morphisms. Then, we conclude that
\[\pi_\F\left(E\right)=\F,\]
 where $\pi_\F: X\times \F\to \F$ is the projection along $\F$.

Since the set of singular points of $\F$ is a strict subvariety, we can consider $f \in \F$ to be a generic point, such that $\pi_\F^{-1}(f)$ intersects transversally $E$ .

$\pi_\F|_{E}:E \longrightarrow \F$ is a dominant fibration, 
hence the generic fibre of $\pi_{\F}$ is reduced (see \cite[Lemma 37.26.4]{the_stacks_project_authors_stacks_2018}), and the intersection between $\pi_{\F}^{-1}(f)$ and $E$ is of multiplicity 1.
Then the multiplicity of the schematic intersection is
\[\mult\left(\pi_{\F}^{-1}(f)\cap (n_{i,j}E)\right) =n_{i,j}.\]
On the other hand, this multiplicity equals the multiplicity of the pull-back $f^\ast(\Delta_i)$, \ie
\[\mult\left(\pi_{\F}^{-1}(f)\cap (n_{i,j}E)\right) =\mult(f^\ast(\Delta_i)).\]
Since $f\in \F$ is an orbifold morphism, 
\[\mult(f^\ast(\Delta_i))\geq m_i.\]
Then, we conclude that 
\[\mult\left(\pi_{\F}^{-1}(f)\cap (n_{i,j}E)\right) =n_{i,j} \geq m_i.\]
\end{proof}

\section{Finiteness result for hyperbolic uniformizable orbifold pairs with semi-negative orbifold canonical bundle}

In this section, we will show a first result which generalizes the theorem proved by Noguchi \cite{noguchi_hyperbolic_1985} in the case of compact hyperbolic manifold with semi-negative first Chern class. This result leads to a metric hypothesis on the orbifold pair.

\begin{defi}\label{defi-unif}
    A smooth orbifold $(X,\Delta)$ is said to be \emph{uniformizable} if there exists an étale $\Delta$-adapted covering $Z\overset{\pi}{\to} (X,\Delta)$, where $Z$ is a smooth variety and $\pi$ is said to be \emph{étale} if it ramifies only over $\Delta$.
\end{defi}

\begin{prop}\label{zero-dim-unif+semipos}
       Let $X$ be a projective variety and $(Y,\Delta)$ be a smooth projective orbifold pair. If $(Y,\Delta)$ is hyperbolic and uniformizable and if the first orbifold Chern class $c_1((Y,\Delta)):=-c_1(K_Y+\Delta)$ is semi-negative, \ie $c_1\left(\left(Y,\Delta\right)\right)\leq 0$, then $\Sur(X,(Y,\Delta))$ is zero-dimensional.
\end{prop}

\begin{proof}
We argue by contradiction. Assume that $\dim \Sur(X,(Y,\Delta))>0$. We follow the proof of \cite[Theorem(4.1)]{noguchi_hyperbolic_1985}. The idea of the proof is to construct a non-zero holomorphic section of some symmetric tensor power of the tangent bundle.
   By the same arguments, we assume that there exists a local deformation $f_t$ of $f$, and so, we can construct a section $\tau \in H^0(Y, S^lTY)$, where $S^l TY$ denotes the $l$-th symmetric tensor power of $TY$ and $l$ is some integer defined in the proof.  
   Let us prove that $p_Y^\ast\tau \in H^0\left(Y, S^lT(p_Y,\Delta)\right)$, where $p_Y:\widetilde{Y} \to (Y,\Delta)$ is the covering given by the uniformization property of $(Y,\Delta)$ (see \cref{def-tang-orbi}). One can note that $T(p_Y,\Delta)=T\widetilde{Y}$ (see \cref{rmq-tangent-orbi}).

We recall all the notations on the following diagram.
   \[\xymatrix{
   f^\ast TY \ar[d]\ar[r]^{\widetilde{f}}& TY\ar[d]^\pi\\
   X \ar[r]_f\ar@/^/[u]^\sigma& (Y,\Delta)
   }\]

Let $y\in Y$ be a generic point in $Y$. We can choose $y$ outside the intersection locus of $\Delta$ since $\Delta$ is a simple normal crossing divisor. Locally, we denote by $y_1,\dots, y_n$ local cordinates on $Y$ centered in $y$ such that $\Delta = \left(1-\frac{1}{m_1}\right)\{y_1=0\}$. Let $x\in X$ such that $f(x)=y$. We denote by $x_1,\dots, x_n$ local coordinates on $X$ centered in $x \in X$.

Since $f$ is an orbifold morphism, $f$ is locally given by
\[\begin{array}{c|ccc}
     f :  & \D^n & \longrightarrow & \D^n \\
     & (x_1,\cdots,x_n) & \longmapsto & (x_1^{p_1},x_2,\cdots,x_n)
\end{array}\]
with $p_1\geq m_1$.
One can assume that the local deformation $f_t$ is given by 
\[\begin{array}{c|ccc}
     F : & \D^n \times \D & \longrightarrow & \D^n \\
     & (x_1,\cdots,x_n;t) & \longmapsto & F(x_1,\cdots,x_n;t)=(F_1(x_1,\cdots,x_n;t),\cdots,F_n(x_1,\cdots,x_n;t)).
\end{array}\]
Considering $y$ sufficiently generic, we can consider that $x$ belongs to only one component of $f^\ast\Delta$, such that
\[F_1(x_1,\cdots,x_n;t)=h_1(x_1,\cdots,x_n;t)^{p_1}\]
and for $j\geq 2$,
\[F_j(x_1,\cdots,x_n;t)=h_j(x_1,\cdots,x_n;t)\]
with $h_1,\dots,h_n$ holomorphic maps and $p_1\geq m_1$.

 We can write the partial derivative as follows
    \[\frac{\partial F_1}{\partial t} (x_1,\dots,x_n;t) = h_{1}(x_1,\dots,x_n;t)^{p_1-1}\psi_{i}(x_1,\dots,x_n;t) \]
and for $j\geq 2$,
\[\frac{\partial F_j}{\partial t} (x_1,\dots,x_n;t) = \psi_{i}(x_1,\dots,x_n;t) \]
    with $\psi_{1},\dots, \psi_n$ some holomorphic maps. 

The section $\sigma \in H^0(X,f^\ast TY)$ is locally given by
    \begin{multline*}
        (x_1,\dots,x_n)\mapsto 
    \left(x_1,\dots, x_n;\psi_{1}(x_1,\dots,x_n;0)h_{1}(x_1,\dots,x_n;0)^{p_1-1},\dots,\right. \\
    \left.\psi_{n}(x_1,\dots,x_n;0)\right).
    \end{multline*}

    The map $\widetilde{f}:f^\ast TY \to TY$ is given by
\[
    (x_1,\dots,x_n;\omega_1,\dots,\omega_n) \mapsto \left( f(x_1,\dots,x_n);\omega_1,\dots,\omega_n\right).
\]
$\pi$ is a finite ramified covering of degree $l$. Taking into account the $l$ preimage of $y\in Y$, we define the section $\tau\in H^0(Y,S^lTY)$ as follow
\[\begin{array}{c|ccc}
     \tau :  & Y & \longrightarrow & S^lTY \\
     & y & \longmapsto & \underset{f(x)=y}{\otimes}\widetilde{f}\circ \sigma (x)
\end{array}\]
Locally, in $y\in Y$, $\tau$ is described by 
\begin{equation*}
\begin{split}
    \tau(y_1,\dots,y_n)&=\underset{f(x)=y}{\otimes}\widetilde{f}\circ \sigma (x)\\
    &=\underset{\substack{x_1^{p_1}=y_1 \\ x_i=y_i,\,i\neq 1}}{\otimes} \widetilde{f}\circ \sigma (x_1,\dots, x_n)\\
    &= \underset{\substack{x_1^{p_1}=y_1 \\ x_i=y_i,\,i\neq 1}}{\otimes}\left(x_1^{p_1},x_2,\dots, x_n\vphantom{h_{1}\left(x_1,\dots,x_n\right)^{p_1-1}}\right.;\\
    &\hspace{1.9cm}\left.\psi_{1}\left(x_1,\dots,x_n;0\right) h_{1}\left(x_1,\dots,x_n;0\right)^{p_1-1},\dots,\psi_{n}\left(x_1,\dots,x_n;0\right)\right)\\
    &= \underset{\substack{x_1^{p_1}=y_1 \\ x_i=y_i,\, i\neq 1}}{\otimes}\left(y_1,y_2,\dots, y_n;\vphantom{h_{1}\left(y_1^{\frac{1}{p_1}},y_2,\dots,y_n;0\right)^{p_1-1}}\right.\\
    &\hspace{2.2cm}\left. \psi_{1}\left(y_1^{\frac{1}{p_1}},y_2,\dots,y_n;0\right)h_{1}\left(y_1^{\frac{1}{p_1}},y_2,\dots,y_n;0\right)^{p_1-1},\dots,\right.\\
    &\hspace{6.9cm}\left. \vphantom{h_{1}\left(y_1^{\frac{1}{p_1}},y_2,\dots,y_n;0\right)^{p_1-1}}\psi_{n}\left(y_1^{\frac{1}{p_1}},y_2,\dots,y_n;0\right)\right).
\end{split}
\end{equation*}
Note that the former expression is well defined since we take the product over all preimages.

Assume that $(Y,\Delta)$ is uniformizable by $\widetilde{Y}$. We represented the situation by the following diagram.

\[\xymatrix{
& {\widetilde{Y}}\ar[d]^{p_Y}\\
X \ar[r]_f & (Y,\Delta)}
\]
Locally, $p_Y$ is given by
\[p_Y: (z_1,\dots,z_m)\mapsto \left(z_1^{m_1},\dots,z_n\right).\]
We can pullback by ${p_Y}$ the above section $\tau$. Then, locally, we obtain a section given by
\begin{equation*}
\begin{split}
p_Y^\ast \tau(z_1,\dots,z_n)&= \tau(p_Y(z_1,\dots,z_n))\\
    &= \tau(z_1^{m_1},z_2,\dots,z_n)\\
    &= \underset{\substack{x_1^{p_1}=z_1^{m_1} \\ x_i=z_i,\, i\neq 1}}{\otimes}\left(z_1^{m_1},z_2,\dots, z_n;\vphantom{h_{1}\left(z_1^{\frac{m_1}{p_1}},z_2,\dots,z_n\right)^{p_1-1}}\right.\\
    &\hspace{1.4cm}\left. \psi_{1}\left(z_1^{\frac{m_1}{p_1}},z_2,\dots,z_n;0\right)h_{1}\left(z_1^{\frac{m_1}{p_1}},z_2,\dots,z_n;0\right)^{p_1-1},\dots,\right.\\
    &\hspace{6.4cm}\left. \vphantom{h_{1}\left(z_1^{\frac{m_1}{p_1}},z_2,\dots,z_n;0\right)^{p_1-1}}\psi_{n}\left(z_1^{\frac{m_1}{p_1}},z_2,\dots,z_n;0\right)\right).
\end{split}
\end{equation*}

So, along $\Delta$, $p_Y^\ast \tau $ vanishes with order at least
\[\frac{m_1}{p_1}(p_1-1)l.\]

Since we consider orbifold morphisms, $m_1\leq p_1$, so
\[\frac{m_1}{p_1}(p_1-1)l\geq (m_1-1)l,\]
where $(m_1-1)l$ corresponds to the multiplicity which appears in the definition of $S^lT(p_Y,\Delta)$ (see \cref{def-tang-orbi}). So, $p_Y^\ast \tau$ is a section of $S^lT(p_Y,\Delta)$ over $\widetilde{Y}$, \ie \[p_Y^\ast \tau \in H^0\left(\widetilde{Y},S^l T(p_Y,\Delta)\right).\]

We denote by $\widehat{Y}$ the universal cover of $\widetilde{Y}$. If there exists $k\in N^\ast$, such that $\widehat{Y}$ is biholomorphic to $\C^k \times Z$, then, by composition, we obtain orbifold entire curves $\C\to (Y,\Delta)$. This contradicts the fact that $(Y,\Delta)$ is hyperbolic, hence there is no euclidean factor in $\widehat{Y}$.

Since $c_1(Y,\Delta)$ is semi-negative, by \cite[Theorem~7]{kobayashi_first_1980}, we see that \[H^0\left(\widetilde{Y},S^l T(p_Y,\Delta)\right)=H^0\left({\widetilde{Y}},S^l T\widetilde{Y} \right)=\{0\},\] while $p_Y^\ast \tau \neq 0$. This is a contradiction and so $\Sur(X,(Y,\Delta))$ is zero-dimensional. 
 \end{proof}

\begin{thm}\label{finite-unif+semipos}
       Let $X$ be a projective variety and $(Y,\Delta)$ be a projective orbifold pair. If $(Y,\Delta)$ is hyperbolic and uniformizable and if $c_1((Y,\Delta))$ is semi-negative, then $\Sur(X,(Y,\Delta))$ is finite.
\end{thm}

\begin{proof}
By \cref{Sur-compact}, it suffices to show that each irreducible component of $\Sur(X,(Y,\Delta))$ is a one-point set.
Then, by \cref{zero-dim-unif+semipos},
one can see that each irreducible component of $\Sur(X,(Y,\Delta))$ is zero-dimensional. Then $\Sur(X,(Y,\Delta))$ contains only a finite number of one-point sets, and so $\Sur(X,(Y,\Delta))$ is finite.
\end{proof}

\section{Automorphism group}

Let $(Y,\Delta)$ be a smooth projective  orbifold. In this section, we are interested in maps of $(Y,\Delta)$ to itself. We refer to \cref{def-orb-map} for the definition of orbifold maps in the general case.
We denote by $\Aut(Y,\Delta)$ the set of orbifold automorphisms which is defined as follows:

\[\Aut(Y,\Delta)=\left\{f\in \Hol((Y,\Delta),(Y,\Delta))\,|\, f \text{ is biholomorphic}\right\}.\]

\begin{lem}\label{def-autom-orbi}
    Let $f\in \Aut(Y,\Delta)$ be an orbifold automorphism of $(Y,\Delta)$. Then $f$ is nothing but an automorphism of $Y$ which permutes the components of $\Delta$ with the same multiplictiy, and the inverse map $f^{-1}$ is also an orbifold map.
\end{lem}

\begin{proof}
Let $f \in \Aut(Y,\Delta)$ with $\Delta=\sum \left(1-\frac{1}{m_i}\right) \Delta_i$. Since $f$ is an automorphism, there exists a unique divisor $D \subset Y$ such that $f^{-1} \Delta_i =D $. Since $f$ is an orbifold map, $m_Y(D)\geq m_Y(\Delta_i)$ where $m_Y(D)$ (resp. $m_Y(\Delta_i)$) denotes the orbifold multiplicity of $D$ (resp. $\Delta_i$) on $Y$.
The map $f$ being invertible, we can do the same with $f^{-1}$ and so, $m_Y(\Delta_i)\geq m_Y(D)$. Thus $m_Y(\Delta_i)=m_Y(D)$.
\end{proof}

\begin{thm}\label{automorphism-group}
    Let $(Y,\Delta)$ be a smooth projective orbifold. If $(Y,\Delta)$ is hyperbolic, then $\Aut(Y,\Delta)$ is finite.
\end{thm}

\begin{proof}
We recall that $\Aut(Y)$ is a complex Lie group since $Y$ is a projective variety (see \cite[Theorem~(5.4.3)]{kobayashi_hyperbolic_1998}).
Furthermore, by \cref{def-autom-orbi}, being an element of $\Aut(Y,\Delta)$ returns to be an element of $\Aut(Y)$ that satisfies some algebraic conditions. So $\Aut(Y,\Delta)$ is an algebraic variety.

Note that $\Aut(Y,\Delta)$ is closed in $\Aut(Y)$. 
Indeed, let $(f_n)$ be a sequence of orbifold automorphisms of $(Y,\Delta)$ which converges to $f\in \Aut(Y)$. Thus, $f(Y)$ cannot be contained in one of the components of $\Delta$. By \cref{morph-orbi-ferme}, we conclude that $f\in \Hol((Y,\Delta),(Y,\Delta))$, hence $f\in \Aut(Y)\cap \Hol((Y,\Delta),(Y,\Delta))$, \ie $f\in \Aut(Y,\Delta)$.

Since $\Aut(Y)$ is closed in $\Hol(Y,Y)$, one can see that $\Aut(Y,\Delta)$ is closed in $\Sur[(Y,\Delta),Y,\Delta)]$. The latter is compact (\cref{Sur-compact}), so $\Aut(Y,\Delta)$ is a compact complex Lie group.

Let us show that its identity component
$\Aut^0(Y,\Delta)$ is trivial. We argue by contradiction, assuming that $\Aut^0(Y,\Delta)$  is not trivial.

Let consider a point $y\in Y\backslash\lceil\Delta\rceil$. By \cref{def-autom-orbi}, $f(y)\notin \lceil\Delta\rceil$ for all $f\in \Aut(Y,\Delta)$. Moreover the evaluation map is an orbifold map (see \cref{ev-orbi}),  so $ev_y : \Aut^0(Y,\Delta) \longrightarrow (Y,\Delta)$ is an orbifold morphism. 

Then, since $\Aut^0(Y,\Delta)$ is an algebraic subgroup, we get a non-constant holomorphic map $\C \to \Aut^0(Y,\Delta)$, and so, by composition, we obtain an orbifold entire curve
\[\C \longrightarrow \Aut^0(Y,\Delta) \overset{ev_y}{\longrightarrow} (Y,\Delta).\]
This map has to be trivial because $(Y,\Delta)$ is hyperbolic. More precisely, $ev_y$ has to be trivial since we assume $\C \to \Aut^0(Y,\Delta)$ not to be constant.
By \cite[Lemma~(5.3.1)]{kobayashi_hyperbolic_1998}, with $\Aut^0(Y,\Delta)$ is compact, $ev:\Aut^0(Y,\Delta)\times \{y\} \to (Y,\Delta)$ is constant for all $y\in Y$.  Thus, for all maps
$f,g \in \Aut^0(Y,\Delta)$, for all $y\in Y$, $f(y)=g(y)$, and so $f=g$.
We conclude that $\Aut^0(Y,\Delta)$ is trivial.

$\Aut(Y,\Delta)$ is compact and zero-dimensional and so we conclude that  $\Aut(Y,\Delta)$ is finite.
\end{proof}

\section{Pointed maps}

Let $(X,\Delta_X)$ and $(Y,\Delta_Y)$ be two orbifold pairs. Let $x$ be a point in $X$ and $y$ be a point in $Y$.
We define the set of holomorphic pointed orbifold maps as follows
\[\Hol\left[((X,\Delta_X),x),\left((Y,\Delta_Y),y\right)\right]=\left\{f\in \Hol((X,\Delta_X) ,(Y,\Delta_Y))\,|\, f(x)=y \right\}.\]

\begin{thm}
Let $(X,\Delta_X)$ be an orbifold pair with $\lfloor \Delta_X\rfloor=\emptyset$. Let $(Y,\Delta_Y)$ be a hyperbolic compact orbifold pair. Let $x$ be a point in $X$ and $y$ be a point in $Y\backslash \lceil\Delta\rceil$.
Then  $\Hol\left[((X,\Delta_X),x),\left((Y,\Delta_Y),y\right)\right]$ is finite.
\end{thm}

\begin{proof}
    First, recall some useful notations. 
   We define the evaluation morphism $ev$ by
\[\begin{array}{c|ccc}
    ev : & X \times \Hol(X,Y) & \longrightarrow & Y \\
     & (x,f) & \longmapsto & f(x)
\end{array}\]
Let $x$ be a point in $X$. We define the morphism $ev_x$, the evaluation at $x$, by
\[\begin{array}{c|ccc}
    ev_x : & \Hol(X,Y) & \longrightarrow & Y \\
     & f & \longmapsto & f(x)
\end{array}\]
$\Hol\left[((X,\Delta_X),x),\left((Y,\Delta_Y),y\right)\right]$ can be interpreted as
\[ev_x^{-1}(\{y\})\cap \Hol[(X,\Delta_X),(Y,\Delta_Y)]  \subset \Hol[(X,\Delta_X),(Y,\Delta_Y)].\]
We denote it by $A$.
Note that $ev_x^{-1}(\{y\})$ is closed in $\Hol(X,Y)$.
We only want to consider orbifold morphisms so we restrict $ev_x$ to $\Hol[(X,\Delta_X),(Y,\Delta_Y)]$.
Since $y \in Y\backslash \lceil\Delta\rceil$, $A$ is closed in $\Hol(X,Y)$ by \cite[Proposition~7]{campana_brody_2009}]. By \cref{Compactness-hol}, $\Hol[(X,\Delta_X),(Y,\Delta_Y)]$ is relatively compact in $\Hol(X,Y)$, then $A$ is compact.

Let $H$ be an irreducible component of $A$. Let us show that $H$ is 0-dimensional.
By definition of $H$, the restriction of $ev$ to $\{x\}\times H$ is constant, equal to $y$. Since $H$ is compact, we can apply \cite[Lemma~(5.3.1)]{kobayashi_hyperbolic_1998}. Then for all point $x'\in X$, the map $\{x'\}\times H \xrightarrow{ev} Y$ is constant. Thus, for all maps $f,g \in H$, for all $x'\in X$, $f(x')=g(x')$. This shows that $f=g$, and so, H is a one-point set.
\end{proof}

\section{Construction of a curve intersecting non-positively the orbifold canonical bundle}\label{section-intersection-non-pos}

In this paragraph, we construct
a curve intersecting non-positively the orbifold canonical bundle.

Let $X$ be a projective variety and $(Y,\Delta)$ be a smooth projective orbifold. As before, we can reduce the study to the case where $\dim X=\dim Y$ and where all the maps from $X$ to $(Y,\Delta)$ are finite. Let $\F\subset \Sur(X,(Y,\Delta))$ be
a smooth projective family of surjective orbifold maps. Since we consider algebraic varieties, we can intersect $\F$ with sufficiently many hyperplanes, such that $\dim \F=1$.

One of the goals of this section is to construct a holomorphic section of the line bundle 
\[ev_x^\ast \left(\left(\displaystyle\bigwedge^{\dim Y} TY-\Delta\right)^{\otimes M}\right)\]
over $\F$, for a well-chosen point $x\in X$ and an integer $M\in \N$.

The main idea is to consider the Jacobian of the evaluation morphism, taken with some power. The power is chosen to give a sense to the orbifold canonical bundle.

\begin{defi}
We define the following sets:
    \[X_0 = \{x\in X\,|\, \dim ev_x(\F) =0\}\]
We choose a generic point $f_0\in \F$, so that the order of each point $x\in X$, satisfying $f(x)\in \lceil \Delta\rceil$, is generic, \ie minimal, and therefore remains the same for each map $f_t\in F$ in the neighbourhood of $f_0$.
We denote the set of these maps by $\F_{\text{min}}$.
    \[\F_{\text{min}} = \{f \in \F \,|\, f \text{ branches over }\Delta \text{ with minimal multiplicity}\}\]
    \[X_1 =\{x\in X\,|\, \exists f\notin \F_{\text{min}}, f(x)\in \Delta\}\]
\end{defi}

\begin{prop}\label{X0-X1-finis}
The sets $X_0$ and $X_1$, defined above, are proper closed sets.
\end{prop}

\begin{proof}
    By construction, $\dim ev_x(\F)\in \{0,1\}$.
    The set $X_0$ is a proper closed set, otherwise, by continuity, all maps of $\F$ are equals at any point of $X$,\ie  $\F$ is a one-point set, which contradicts the fact that $\dim \F=1$. 

By hypothesis, $\F$ is a curve. We are interested in the points of $\F$ that branch along each divisor $\Delta_i\in \lceil \Delta\rceil$, with minimal multiplicity. There exists at most a finite number of non generic points, so $\overline{\F_{\text{min}}}$ is a finite set.

Since each orbifold morphism of $\F$ is continuous, $f^{-1}(\Delta)$ is a closed set and so, $X_1$ is a closed set.
\end{proof}

\begin{thm}\label{construction-section-sur-F}
    Let $X$ be a projective variety and $(Y,\Delta)$ be a smooth compact orbifold. Assume that $\dim X=\dim Y$. If there exists a family $\F\subset \Sur(X,(Y,\Delta))$ of orbifold morphisms, such that $\dim \F=1$, then we can construct a holomorphic non-zero section $\sigma$ of 
    $ev_x^\ast ((\bigwedge^{\dim Y} TY-\Delta)^{\otimes M})$ on $\F$, for every general point $x \in X$ and an integer $M\in \N$, \ie
    \[\sigma \in H^0\left(\F,ev_x^\ast \left(\left(\displaystyle\bigwedge^{\dim Y} TY-\Delta\right)^{\otimes M}\right)\right).\]
\end{thm}

\begin{proof}
By \cref{X0-X1-finis}, $X\backslash \left(X_0\cup X_1\right)$ is a non-empty open set. We consider a point $x\in X\backslash \left(X_0\cup X_1\right)$.

We set $M=\lcm\{m_i\}$, where $m_i$ are the multiplicities in $\Delta=\sum_i \left(1-\frac{1}{m_i}\right) \Delta_i$. This choice of $M$ permits to contruct the bundle
    \[\left(\displaystyle\bigwedge^{\dim Y} TY -\Delta\right)^{\otimes M}.\]

By seting $\F=\{f_t:X\to (Y,\Delta)\}$, the data of $t$ defines local coordinates on $\F$ centered in $f_0$.
We give local coordinates $z_1,\dots, z_n$ on $X$ centered in $x \in X$, and $w_1,\dots, w_n$ on $Y$ centered in $y=f_0(x)$.

    We assume that $\Delta$ has $l$ distinct components around $y=f_0(x)$. We consider simple normal crossing divisor, so locally we can write \[\Delta = \left(1-\frac{1}{m_1}\right)\{w_1=0\}+\cdots + \left(1-\frac{1}{m_l}\right)\{w_l=0\}\]
    for a well-chosen system of local coordinates.
The evaluation morphism $ev$
\[\begin{array}{c|ccc}
    ev : & X \times \F & \longrightarrow & Y \\
     & (x,f) & \longmapsto & f(x)=(f_1(x),\cdots,f_n(x))
\end{array}\]
is given, in local coordinates, by 
\[\begin{array}{c|ccc}
    F : & \D^n \times \D & \longrightarrow & \D^n \\
     & (z_1,\cdots,z_n;t) & \longmapsto & F(z_1,\cdots,z_n;t)=(F_1(z_1,\cdots,z_n;t),\cdots,F_n(z_1,\cdots,z_n;t)).
\end{array}\]

    The pull-back by $F$ of a component $\Delta_j=\{w_j=0\}\in\lceil \Delta\rceil$ is given by
    \[F^\ast \left\{w_j=0\right\}= \sum_i n_{i,j} E_{i,j}\]
    with $n_{i,j}\geq m_j$ because the evalutation morphism $F$ is an orbifold morphism (\cref{ev-orbi}).
  
    We describe locally $E_{i,j}$ by setting \[E_{i,j}=\{h_{i,j}(z_1,\dots,z_n;t)=0\},\] 
    with $h_{i,j}$ a holomorphic map.

    If $x$ belongs to $k_{j}$ components of $F^\ast \left\{w_j\right\}$, one can write
    \[F_j(z_1,\dots,z_n;t)= \varphi_j(z_1,\dots,z_n;t) \prod_{i=1}^{k_{j}} h_{i,j}(z_1,\dots,z_n;t)^{n_{i,j}},\]
    with $\varphi_{j}$ a holomorphic map.

    We can write the partial derivative as follows
    \[\frac{\partial F_j}{\partial z_r} (z_1,\dots,z_n;t) = \psi_{j,r}(z_1,\dots,z_n;t)\prod_{i=1}^{k_{j}} h_{i,j}(z_1,\dots,z_n;t)^{n_{i,j}-1}, \]
    with $\psi_{j,r}$ some holomorphic maps.

    In local coordinates, the Jacobian of $F$ is given by 

\begin{align*}
    &JF(z_1,\dots,z_n;t)(e_1\wedge \dots \wedge e_n) =  \sum_{\nu \in \mathfrak{S}_n} \varepsilon(\nu) \prod_{s=1}^n \frac{\partial F_s}{\partial z_{\nu(s)}} (z_1,\dots,z_n;t)\\
         &=  \sum_{\nu \in \mathfrak{S}_n} \varepsilon(\nu) \prod_{s=1}^{l}\frac{\partial F_s}{\partial z_{\nu(s)}} (z_1,\dots,z_n;t)  \prod_{s=l+1}^n \frac{\partial F_s}{\partial z_{\nu(s)}} (z_1,\dots,z_n;t)\\
         \begin{split}
             &=\sum_{\nu \in \mathfrak{S}_n} \left[\varepsilon(\nu) \prod_{s=1}^{l}\left(\psi_{s,\nu(s)}(z_1,\dots,z_n;t)
        \prod_{i=1}^{k_{s}} h_{i,s}(z_1,\dots,z_n;t)^{n_{i,s}-1}\right)\right.\\
        & \hspace*{2cm} \times  \left.\prod_{s=l+1}^n \frac{\partial F_s}{\partial z_{\nu(s)}} (z_1,\dots,z_n;t)\right]
         \end{split}\\
\begin{split}
    &= \prod_{s=1}^{l}\left(\prod_{i=1}^{k_{s}} h_{i,s}(z_1,\dots,z_n;t)^{n_{i,s}-1}\right)\\
    &\hspace*{2cm} \times \sum_{\nu \in \mathfrak{S}_n} \left[ \varepsilon(\nu) \prod_{s=1}^{l}\psi_{s,\nu(s)}(z_1,\dots,z_n;t) \prod_{s=l+1}^n \frac{\partial F_s}{\partial z_{\nu(s)}} (z_1,\dots,z_n;t)\right]
\end{split}
\end{align*} 

We define the section $\sigma$ by $JF(z_1,\dots,z_n;t)(e_1\wedge \dots \wedge e_n)^{\otimes M}$. In local coordinates, its expression is 
    \begin{multline}\label{coord-loc-JFM}
        \prod_{s=1}^{l}\left(\prod_{i=1}^{k_{s}} h_{i,s}(z_1,\dots,z_n;t)^{n_{i,s}-1}\right)^M\\
    \times \left(\sum_{\nu \in \mathfrak{S}_n} \left[ \varepsilon(\nu) \prod_{s=1}^{l}\psi_{s,\nu(s)}(z_1,\dots,z_n;t) \prod_{s=l+1}^n \frac{\partial F_s}{\partial z_{\nu(s)}} (z_1,\dots,z_n;t)\right]\right)^{\otimes M}.
    \end{multline}

In the chosen system of coordinates, $\left(\displaystyle\bigwedge^{\dim Y} TY-\Delta\right)^{\otimes M}$ is generated by 
\[w_1^{\left(1-\frac{1}{m_1}\right) M}\cdots w_{l}^{\left(1-\frac{1}{m_{l}}\right) M}\left(\frac{\partial}{\partial w_1}\wedge \cdots \wedge\frac{\partial}{\partial w_{n}}\right)^{\otimes M}.\]
Recall that this line bundle is well defined because $M$ is defined by $M=\lcm\{m_i\}$, where $m_i$ are the multiplicities in $\Delta=\sum_i \left(1-\frac{1}{m_i}\right) \Delta_i$. 

The pull-back bundle $F^\ast \left(\displaystyle\bigwedge^{\dim Y} TY-\Delta\right)^{\otimes M}$ is given by
\begin{equation*}
\begin{split}
    &F_1(z_1,\dots,z_n;t)^{\left(1-\frac{1}{m_1}\right) M}\cdots F_{l}(z_1,\dots,z_n;t)^{\left(1-\frac{1}{m_{l}}\right) M} F^\ast \left(\frac{\partial}{\partial w_1}\wedge \cdots \wedge\frac{\partial}{\partial w_{n}}\right)^{\otimes M}\\
    &=\prod_{j=1}^{l} F_j(z_1,\dots,z_n;t)^{\left(1-\frac{1}{m_j}\right) M}F^\ast \left(\frac{\partial}{\partial w_1}\wedge \cdots \wedge\frac{\partial}{\partial w_{n}}\right)^{\otimes M}\\
    &=\prod_{j=1}^{l} \left(\varphi_j(z_1,\dots, z_n;t)^{\left(1-\frac{1}{m_j}\right) M}\prod_{i=1}^{k_{j}} h_{i,j}(z_1,\dots,z_n;t)^{n_{i,j}\left(1-\frac{1}{m_j}\right) M}\right) \\
    &\hspace{8.3cm} F^\ast \left(\frac{\partial}{\partial w_1}\wedge \cdots \wedge\frac{\partial}{\partial w_{n}}\right)^{\otimes M}.
\end{split}
\end{equation*} 
Since $n_{i,j}\geq m_j$, we have
\[n_{i,j}\left(1-\frac{1}{m_j}\right) M = \left(n_{i,j}-\frac{n_{i,j}}{m_j}\right)M \leq (n_{i,j}-1)M,\]
where $(n_{i,j}-1)M$ corresponds to the multiplicity in the equation~\lref{coord-loc-JFM}.

Thus, we conclude that
\[\sigma \in H^0\left(\F, ev_x^\ast \left(\displaystyle\bigwedge^{\dim Y} TY-\Delta\right)^{\otimes M}\right).\]
\end{proof}

As a direct consequence, we have the following result.

\begin{prop}\label{intersection-non-pos-fort}
    Let $X$ be a projective variety and $(Y,\Delta)$ be a smooth compact orbifold. Assume that $\dim X=\dim Y$. Assume that there exists a smooth projective 1-dimensional family $\F$ such that there exists a non-constant holomorphic map $\varphi:\F\to \Sur(X,(Y,\Delta))$. Then for every general point $x\in X$, the curve $ev_x \varphi(\F)$ intersects non positively the orbifold canonical bundle.
\end{prop}

\begin{proof}
Since there exists a non-constant holomorphic map $\varphi:\F\to \Sur(X,(Y,\Delta))$, we obtain a curve $\varphi(\F) \subset \Sur(X,(Y,\Delta))$.
    We apply \cref{construction-section-sur-F}. We can construct a section 
\[\sigma \in H^0\left(\varphi(\F), ev_x^\ast \left(\displaystyle\bigwedge^{\dim Y} TY-\Delta\right)^{\otimes M}\right),\]
for a well-chosen integer $M\in \N$ and $x\in X$, a well-chosen point.

Hence the degree of the line bundle $ev_x^\ast\left(\bigwedge^{\dim Y} TY-\Delta\right)^{\otimes M}$ is non negative, \ie
    \begin{align*}
\deg_{\varphi(\F)} ev_x^\ast \left( \bigwedge^{\dim Y} TY-\Delta\right)^{\otimes M}\geq0.\\
\intertext{By definition $\bigwedge^{\dim Y} TY^\ast = K_Y$, and using additive notations, we have}
\deg_{\varphi(\F)} ev_x^\ast \left(-( K_Y+\Delta)\right)^{M}\geq 0\\
\intertext{By the projection formula, then we obtain}
\deg_{ev_x({\varphi(\F)})}\left(-( K_Y+\Delta)\right)^{M}\geq 0,
\end{align*}
which can be restated as
\begin{equation}
    \deg_{ev_x({\varphi(\F)})}\left( K_Y+\Delta\right)^{M}\leq0.
\end{equation}

Therefore $ev_x({\varphi(\F)})$ intersects non positively $(K_Y+\Delta)^{\otimes M}$.
\end{proof}

\section{Finiteness result for hyperbolic orbifold pairs with big orbifold canonical bundle}

In this section, we will prove a new rigidity result based on analytic tools. More precisely, we consider a positivity condition on the orbifold canonical bundle, namely the fact that it is big.

This result was already proved in \cite{bartsch_kobayashi-ochiais_2024}. Here, we use analytic methods.

\begin{prop}\label{zero-dim-big}
 Let $X$ be a projective variety. Let $(Y,\Delta)$ be a smooth compact orbifold.
 If $(Y,\Delta)$ is hyperbolic and if the orbifold canonical bundle $K_Y+\Delta$ is big, then the set of surjective holomorphic orbifold morphisms from $X$ to $(Y,\Delta)$ is zero-dimensional.
\end{prop}

\begin{proof}
We argue by contradiction. Assume that there exists an irreducible component $\F$ of $\Sur(X,(Y,\Delta))$, such that $\dim \F\geq 1$.
Since we consider algebraic varieties, we can intersect $\F$ with sufficiently many hyperplanes, in such a way that $\dim \F= 1$.
By  the characterization of big divisors (see \cite[Corollary~2.2.7.]{lazarsfeld_positivity_2004-1}), if the orbifold canonical divisor $K_Y+\Delta$ is big, then there exists an ample divisor $A$, a positive integer $m>0$, and an effective divisor $E$ such that $m(K_Y+\Delta)\equiv_{\text{num}}A+E$. 
Therefore, \[m(K_Y+\Delta)\cdot C = m(A+E)\cdot C= mA\cdot C + mE\cdot C,\] where $C=ev_x(\F)$, $x$ being a general point.

Since $A$ is ample, $A\cdot C>0$. Let us show, by adding restriction on the choice of the point $x\in X$ define in \cref{construction-section-sur-F}, that we can have $E\cdot C\geq 0$.
Let $y \notin E$. There exists a surjective map $f$ such that $f^{-1}({y})\neq \emptyset$. Let $x\in f^{-1}({y})$. In the proof of \cref{construction-section-sur-F}, we gave conditions about the choice of $x$. More precisely, we chose $x$ in the open set $X\backslash (X_0\cup X_1)$ (see \cref{X0-X1-finis}). Here, we can choose $x\in X$ satisfying the same conditions.

\[m(K_Y+\Delta)\cdot C = m(A+E)\cdot C= m\underbrace{A\cdot C}_{>0} + m\underbrace{E\cdot C}_{\geq 0}> 0.\]
It leads to a contradiction with \cref{intersection-non-pos-fort}, and thus we conclude the proof of \cref{zero-dim-big}.
\end{proof}

\begin{thm}\label{finite-big}
       Let $X$ be a projective variety and $(Y,\Delta)$ be a compact orbifold pair. If $(Y,\Delta)$ is hyperbolic and if $K_Y+\Delta$ is big, then $\Sur(X,(Y,\Delta))$ is finite.
\end{thm}

\begin{proof}
By \cref{Sur-compact}, it suffices to show that each irreducible component of $\Sur(X,(Y,\Delta))$ is a one-point set.
Then, by \cref{zero-dim-big},
one can see that each irreducible component of $\Sur(X,(Y,\Delta))$ is zero-dimensional. Then $\Sur(X,(Y,\Delta))$ contains only a finite number of one-point sets, and so $\Sur(X,(Y,\Delta))$ is finite.
\end{proof}

\section{Curves intersecting negatively the orbifold canonical bundle}\label{section-inters-neg}

In this section, we will show that curves constructed at the end of \cref{section-intersection-non-pos} do not only intersect non positively the orbifold canonical bundle but even negatively.

In this section, we still consider $X$ a projective variety and $(Y,\Delta)$ a smooth compact orbifold pair, with $\Delta=\sum\left(1-\frac{1}{m_i}\right) \Delta_i$. Let $\F$ be an irreducible component of $\Sur(X,(Y,\Delta))$, the set of surjective orbifold maps.

We define the degeneracy locus $\mathcal{D}$ by the zero locus of the Jacobian as a section of the determinant of the orbifold tangent bundle.
\[\mathcal{D}_{0,\text{orb}} = \left\{(x, g) \in X \times \F\, \left|\, \, \sigma_g(x)=0\right.\right\},\]
where $\sigma_g\in H^0\left(X,\left(\bigwedge^{\dim Y} TY-\Delta\right)^{\otimes M} \right)$, defined by $\sigma_g(x)=Jg(x)^{\otimes M}$ for $(x,g)\in X\times \F$, and
\[\mathcal{D}_\text{orb}=\overline{\mathcal{D}_{0,\text{orb}}}\subset X\times \F.\]
 We use the natural projection $\pi_X$ and the evaluation map $ev$:
\[\pi_X:X \times \F \to X  \quad\text{and} \quad
ev:X \times \F \to (Y,\Delta).\]
Let us prove an important lemma.

\begin{lem}\label{degeneracy-locus}
If $(Y,\Delta)$ is hyperbolic, then $\pi_X(\mathcal{D}_\text{orb})=X$.
\end{lem}

Before, we prove this lemma.
\begin{lem}\label{lemma-champs-vecteurs-log}
    Let $Y$ be a compact variety and $D$ be a divisor. If $(Y,D)$ is hyperbolic, then $H^0(Y,T_Y(-\log D))=\{0\}$.
\end{lem}

\begin{proof}
    We argue by contradiction. Let $\xi \in H^0(Y,TY(-\log D))$, $\xi \neq 0$. We consider $\xi$ as a section of $TY$, \ie $\xi \in H^0(Y,TY)$. We recall that every vector field on a compact variety is complete. Then, since $Y$ is compact, $\xi$ is complete, and so the flow gives entire curves on $Y$.
    We denote by $\gamma$ one of this entire curve, which satisfies
    \[\left\{\begin{array}{rl}
        \gamma(t)&=\xi (\gamma(t))\\
        \gamma(0)&=y_0
    \end{array}\right.\]
    with $y_0\in Y$. 
    Since $\xi \in H^0(Y,TY(-\log D))$, if we consider $y_0\notin D$, $\gamma(t)\notin D$ for all $t\in \C$. Therefore, the flow defines an entire curve 
    \[\gamma:\C\to Y\backslash D.\]
    Since $Y\backslash D$ is hyperbolic, we obtain a contradiction.
\end{proof}

In the non-orbifold case, \cref{degeneracy-locus} is proved in \cite[Lemma~(6.6.4)]{kobayashi_hyperbolic_1998}.
Now, we adapt the proof in the orbifold case.

\begin{proof}[Proof of \cref{degeneracy-locus}]
Fix $f\in \F$. At the beginning of Section 3 of Chapter 5 of \cite{kobayashi_hyperbolic_1998}, Kobayashi defined an injective map $\nu_f:T_f\F \to H^0(X,f^\ast TY)$. If $f_t$ is a curve in $\F$ such that $f_0=f$ and if $\zeta\in T_f\F$ is a nonzero vector defined by $\zeta=\left.\frac{\partial f_t}{\partial t}\right|_{t=0} \in T_f\F$, then \[(\nu_f(\zeta))(x)=\left.\frac{\partial f_t}{\partial t}\right|_{t=0} \in T_{f(x)}\F.\]
More formally, using the differential
\[ev_*: TX\times T\F \to TY\]
of the evaluation map $ev$, we set 
\[(\nu_f(\zeta))(x)=ev_*(0_x,\zeta)\in T_{f(x)}Y, \quad x\in X, \zeta \in T_f\F,\]
where $0_x$ stands for the zero vector at $x$.
We set $\widetilde{\zeta}=\nu_f(\zeta) \in H^0(X,f^\ast TY)$. The section $\widetilde{\zeta}$ of $f^*TY$
defines a multi-valued section $\widehat{\zeta}$ of $TY$, \ie
\[\widehat{\zeta}(y)=\left\{\widetilde{\zeta}(x)\in T_yY \,|\, x\in f^{-1}(y)\right\}.\]

Let $\widehat{Y} = \widehat{\zeta}(Y)\subset TY$. Then the natural projection $\widehat{Y}\to Y$ is a finite 
surjective map. Let $\widetilde{ev}: ev^\ast TY \to TY$ be the natural map (which identifies the fibre $T_{g(x)}Y$
of $ev^\ast TY$ at $(x, g) \in X \times \F$ with the fibre $T_{g(x)}Y$ of TY at $g(x) \in Y$). Then
$\widetilde{ev}^{-1}(\widehat{Y}) \to X \times \F$ is a finite surjective map. We may view $\widetilde{ev}^{-1}(\widehat{Y})$ as a multivalued
section of $ev^\ast TY$ over $X \times \F$.

On the other hand, viewing $f^\ast TY$ as the restriction of the pull-back bundle
$ev^\ast TY$ to $X \times \{f\}$ we consider $\widetilde{\zeta}$ as a section of $ev^\ast TY$ over $X \times \{f\}$, and we
shall extend $\widetilde{\zeta}$ to a holomorphic section $\widetilde{\zeta}$ of $ev^\ast TY$ over $(X \backslash A) \times \F$, where
\[A = \{x \in X \,|\, f(x) \in Y, \sigma_f(x) = 0\} .\]
We define $v(x) = f_\ast^{-1}(\widetilde{\zeta}(x)) \in T_xX$ for $x \in X \backslash A$ and
\begin{equation}\label{6.6.5}
\widetilde{\zeta}(x, g) = g_\ast(v(x)) \in T_{g(x)}Y, \quad (x, g) \in  (X \backslash A) \times \F.
\end{equation}
Thus we have a single valued section $\widetilde{\zeta}$ of $ev^\ast TY$ defined only on $(X \backslash A) \times \F$,
as well as a multi-valued section $\widetilde{ev}^{-1} (\widehat{Y})$ of $ev^\ast TY$ defined on the whole of $X \times \F$.

Now, suppose that $\pi_X(\mathcal{D}_{\text{orb}})\neq X$, and set
\[B = \pi_X(\mathcal{D}_{\text{orb}}).\]
Clearly, $A \subset B$. Then for every $g\in \F$, the section $\sigma_g$ is nonzero in $x\in B$, \ie each map $g\in F$ is orbi-étale in each point $x\in B$. Hence, by \cref{ev-orbi}, one can see that $ev$ is orbi-étale on $B\times \F$. So we can consider the bundle $T(ev,\Delta)$. 
Since $T_{g(x)}Y$ is the fibre of the pull-back bundle $ev^\ast TY$ at $(x, g)\in X \times \F$, we see that
$\left.T(ev,\Delta)\right|_{\{x\}\times \F}$ is isomorphic to the product bundle over $\{x\} \times \F$ with fibre $T_x X$,
\ie
\begin{equation}\label{6.6.6}
  \left.T(ev,\Delta)\right|_{\{x\}\times \F} \simeq \left(\{x\} \times \F\right) \times T_x X,\quad X\in X\backslash B.  
\end{equation}

Fix $x \in X \backslash B$. Since the bundle $T(ev,\Delta)$ restricted to $\{x\} \times \F$ is a product bundle
$(\{x\} \times \F) \times T_x X$ by \lref{6.6.6} and since $\{x\} \times F$ is compact, the restriction of the mutlivalued
section $\widetilde{ev}^{-1}(\widehat{Y})$ to $\{x\} \times \F$ consists of constant sections $(\{x\} \times \F) \times \sigma_i (x)$,
$i = 1, \dots , m,$ where $\sigma_i(x) \in T_xX$ is independent of $g \in \F$.

On the other hand, since $\widetilde{\zeta}(x) \subset \widehat{\zeta}(f(x))$, $\widetilde{\zeta}(X)$ is contained in the mutlivalued
section $\widetilde{ev}^{- 1} (\widehat{Y})$. So by renumbering $\sigma_1, \dots ,\sigma_m$ we may assume that
\[\widetilde{\zeta}(x) = (x, f, \sigma_1 (x)) \in \{x\} \times \F \times T_x X,\quad x \in X\backslash B.\]
Since the extension $\widetilde{\zeta}$ of $\widetilde{\zeta}$ to $(X \backslash B) \times \F$ defined by \lref{6.6.5} is nothing but the
extension of $\widetilde{\zeta}$ by the trivialization \lref{6.6.6}, we have
\[\{(x, g) = (x, g, \sigma_1(x)) \in \{x\} \times \F \times T_x X,\quad x \in X\backslash B.\]
Since $\widetilde{\zeta}((X \backslash B) \times \F)$ is contained in $\widetilde{ev}^{-1} (\widehat{Y})$, $\widetilde{\zeta}$ extends to $X \times \F$.
We denote this extended section of $ev^\ast TY$ over $X \times \F$ by the same symbol
$\widetilde{\zeta}$. For each $y \in Y$, we consider a mapping $ev^{-1}(y) \longrightarrow T_y Y$ which sends $(x, g) \in
ev^{-1}(y)$ to $g_\ast(\widetilde{\zeta}(x,g)) \in T_yY$. This map is constant since $ev^{-1}(y)$ is connected
and compact. Thus we obtain a non-trivial holomorphic vector field on $Y$.

We denote it by $\rho \in H^0(Y,TY)$, defined by \[\rho(y)=g_\ast \left(\widetilde{\zeta}(x,g)\right)\] with $(x,g)\in ev^{-1}(y)$, for each $y\in Y$.

In this case, this pushforward is not only a holomorphic vector field on $Y$ but also a section of logarithmic tangent bundle $TY(-\log \lceil \Delta\rceil)$ over $Y$, \ie \[\rho(y) \in H^0(Y, TY(-\log \lceil \Delta\rceil)).\]

Let us prove this. Let $x$ be a point in $X$ and $g \in \F$.
We give local coordinates $z_1,\dots, z_n$ on $Y$ centered in $y=g(x)\in Y$.

    We assume that $\Delta$ has $l$ distinct components around $y=g(x)$. We consider simple normal crossing divisor, so locally we can write \[\Delta = \left(1-\frac{1}{m_1}\right)\{z_1=0\}+\cdots + \left(1-\frac{1}{m_l}\right)\{z_l=0\}\]
    for a well-chosen system of local coordinates, so that $\lceil\Delta\rceil$ is given by $\{z_1\cdots z_l=0\}$.
    
    Then, the logarithmic tangent bundle $TY(-\log \lceil \Delta\rceil)$ is generated by \[\left \langle z_1 \frac{\partial}{\partial z_1}, \dots, z_l \frac{\partial}{\partial z_l}, \frac{\partial}{\partial z_{l+1}},\dots,  \frac{\partial}{\partial z_n} \right \rangle.\]

 By construction, the section $\widetilde{\zeta}$ is given by some combination of elements of the form $g^\ast \frac{\partial}{\partial z_i}$.
The map $g$ is an orbifold map, so we may assume that $g$ is locally given by 
    $(z_1,\cdots, z_n)\mapsto (z_1^{n_1}, \dots, z_l^{n_l},z_{l+1}, \dots, z_n)$, with $n_i\geq m_i$. 

Since $g_\ast \widetilde{\zeta}$ is nothing but the differential of $g$ along $\widetilde{\zeta}$, then we can compute it in the following way:
    \[ d g\left(g^\ast \frac{\partial}{\partial z_i}\right)= m_i z_i^{m_i-1} \frac{\partial}{\partial z_i} .\]
    Thus if $m_i \geq 2$, this complete the proof that the section $\rho$ is a non-zero section of the logarithmic tangent bundle $TY(-\log \lceil \Delta \rceil)$ over $Y$. Since $(Y,\Delta)$ is hyperbolic, so is $(Y,\lceil \Delta\rceil)$. By \cref{lemma-champs-vecteurs-log},  $H^0(Y,TY(-\log \lceil \Delta \rceil))= \{0\}$, and thus we get a contradiction.
     \end{proof}

This lemma gives that the section, constructed in \cref{construction-section-sur-F}, has zeros. Hence the degree of the line bundle $ev_x^\ast\left(\bigwedge^{\dim Y} TY-\Delta\right)^{\otimes M}$ is positive. Therefore, we can strengthen the previous result (\cref{intersection-non-pos-fort}), improving the proof  with strict inequalities.

\begin{prop}\label{inters-neg-fort}
    Let $X$ be a projective variety and $(Y,\Delta)$ be a smooth compact orbifold. Assume that $\dim X=\dim Y$. Assume that there exists a smooth projective 1-dimensional family $\F$ such that there exists a non-constant holomorphic map $\varphi:\F\to \Sur(X,(Y,\Delta))$. If $(Y,\Delta)$ is hyperbolic, then for every general point $x\in X$, the curve $ev_x \varphi(\F)$ intersects negatively the orbifold canonical bundle.
\end{prop}

\section{Finiteness results}

\subsection{Finiteness result for hyperbolic uniformizable orbifold pairs}

In this section, we will prove a new rigidity result based on an algebraic theorem proved by Miyaoka–Mori in the case of uniformizable orbifold pairs. We recall that the definition of an uniformizable orbifold pair is given by \cref{defi-unif}.

\begin{prop}\label{zero-dim-unif}
Let $X$ be a projective variety and $(Y,\Delta)$ be a smooth projective orbifold pair. If $(Y,\Delta)$ is hyperbolic and uniformizable, then $\Sur(X,(Y,\Delta))$, the set of surjective orbifold maps from $X$ on $(Y,\Delta)$, is zero-dimensional.
\end{prop}

The result is based on the following theorem of Miyaoka–Mori \cite{miyaoka_numerical_1986}.
\begin{thm}\label{KMM}
    Let $X$ be a non-singular projective algebraic manifold, $C$ a closed curve on $X$ and $x$ a general point of $C$. If $K_X \cdot C <0$, then there exists a rational curve $L$ through $x$.
\end{thm}

\begin{proof}[Proof of \cref{zero-dim-unif}]
We argue by contradiction. Assume that there exists an irreducible component $\F$ of $\Sur(X,(Y,\Delta))$, such that $\dim \F\geq1$.
Since we consider algebraic varieties, we can intersect $\F$ with sufficiently many hyperplanes, in such a way that $\dim \F= 1$.
By \cref{inters-neg-fort}, there exists a point $x\in X$ and an integer $M\in \N$, such that the curve $C=ev_x(\F)$ intersects negatively $(K_Y+\Delta)^{\otimes M}$. Note that $C\not\subset \Delta$ because we consider a subset $\F$ of surjective maps.

Since the orbifold is uniformizable, there exists an étale $\Delta$-adapted covering $Z$.
\[\xymatrix{
    & Z \ar[d]^\pi \\
    \F \ar[r]_{ev_x} & (Y,\Delta)
  }\]
   The covering being étale, we have $\pi^\ast\left(K_Y+\Delta\right)=K_Z.$ Then, one can lift the curve $C=ev_x(\F)$ in $Z$.  So, \[\pi^{-1}(C)\cdot K_Z <0.\]
  
  By \cref{KMM}, there exists a map $g:\PP^1\to Z$ such that the curve $g(\PP^1)$ passes through any general point of $\pi^{-1}(C)$. Fix $z\in \pi^{-1}(C)\subset Z$ such a general point. We can chose $z$ such that $\pi(z)\notin \Delta$ since $C\not\subset \Delta$.  If we compose by $\pi$, which is étale, we obtain an orbifold morphism 
  \begin{equation}\label{absurdite-hyp-unif}
      \pi\circ g: \PP^1 \to (Y,\Delta).
  \end{equation}
   Note that $\pi\circ g(\PP^1)\not\subset \Delta$ because the curve $\pi\circ g(\PP^1)$ passes through $\pi(z)\notin \Delta$. 
Then, $\pi\circ g(\PP^1)$ is a rational orbifold curve on $(Y,\Delta)$ which is a contradiction with the fact that $(Y,\Delta)$ is hyperbolic. This conclude the proof of \cref{zero-dim-big}.
\end{proof}

\begin{thm}\label{finite-unif}
       Let $X$ be a projective variety and $(Y,\Delta)$ be a projective orbifold pair. If $(Y,\Delta)$ is hyperbolic and uniformizable, then $\Sur(X,(Y,\Delta))$ is finite.
\end{thm}

\begin{proof}
By \cref{Sur-compact}, it suffices to show that each irreducible component of $\Sur(X,(Y,\Delta))$ is a one-point set.
Then, by \cref{zero-dim-unif},
one can see that each irreducible component of $\Sur(X,(Y,\Delta))$ is zero-dimensional. Then $\Sur(X,(Y,\Delta))$ contains only a finite number of one-point sets, and so $\Sur(X,(Y,\Delta))$ is finite.
\end{proof}

\begin{rmq}
    Obviously, this result implies \cref{finite-unif+semipos}.
\end{rmq}

\begin{ex}[Compact quotient of the unit ball] All Riemann surfaces are characterized by the Poincaré uniformization Theorem. In particular, hyperbolic compact Riemann surfaces are quotients of the unit disc. By De Franchis Theorem \cite{de_franchis_teorema_1913}, there is only a finite number of surjective maps with values in a quotient of the unit disc.
In higher dimension, if we consider compact quotients of the unit ball, they describe some uniformizable hyperbolic orbifolds. By \cref{finite-unif}, there is only a finite number of surjective orbifold maps with values in a quotient of the unit ball.
\end{ex}

\subsection{Finiteness result for hyperbolic orbifold pairs with nef orbifold canonical bundle}

In this subsection, we will prove a new rigidity result based on analytic condition. More precisely, we consider a positivity condition on the orbifold canonical bundle, namely the fact that it is nef.

\begin{prop}\label{zero-dim-nef}
    Let $X$ be a projective variety and $(Y,\Delta)$ be a compact orbifold pair. If $(Y,\Delta)$ is hyperbolic and if the orbifold canonical bundle $K_Y+\Delta$ is nef, then $\Sur(X,(Y,\Delta))$, the set of surjective orbifold maps from $X$ on $(Y,\Delta)$, is zero-dimensional.
\end{prop}

\begin{proof}
The idea of the proof is the same as the proof of \cref{zero-dim-big}, just using \cref{inters-neg-fort} instead of \cref{intersection-non-pos-fort}.

We argue by contradiction. Assume that $\dim \Sur(X,(Y,\Delta))>0$, \ie there exists an irreducible component $\F$ of $\Sur(X,(Y,\Delta))$, such that $\dim \F\geq1$. Since we consider algebraic varieties, we can intersect $\F$ with sufficiently many hyperplanes, in such a way that $\dim \F= 1$.
By \cref{inters-neg-fort}, there exist a point $x\in X$ and an integer $M\in \N$, such that the curve $ev_x(\F)$ intersects negatively $(K_Y+\Delta)^{\otimes M}$.
Then, we obtain a direct contradiction with the fact that $K_Y+\Delta$ is nef and thus we conclude the proof of \cref{zero-dim-nef}.
\end{proof}

\begin{thm}\label{finite-nef}
       Let $X$ be a projective variety and $(Y,\Delta)$ be a compact orbifold pair. If $(Y,\Delta)$ is hyperbolic and  if the orbifold canonical bundle $K_Y+\Delta$ is nef, then $\Sur(X,(Y,\Delta))$ is finite.
\end{thm}

The proof of this result is the same as the proof of \cref{finite-unif}, except that \cref{zero-dim-unif} is replaced by \cref{zero-dim-nef}.

\begin{rmq}
    Obviously, this result implies \cref{finite-big}, but here we need the negativity result of the intersection.
\end{rmq}

\subsection{Finiteness result for hyperbolic orbifold pairs with pseudo-effective orbifold canonical bundle}

In this last subsection, we prove a more general rigidity result. We consider another kind of positivity condition on the orbifold canonical bundle, namely the fact that it is pseudo-effective.

Based on \cite{lazarsfeld_positivity_2004-1}, we recall the definition of pseudo-effective divisors.

\begin{defi}
Let $X$ be a variety.
    The \emph{pseudo-effective cone}
$\mathrm{Pseff}(X)\subset N^1_\R(X)$ 
is the closure of the convex cone spanned by the classes of all effective $\R$-divisors.
A divisor $D \in \Div_\R(X)$ is \emph{pseudo-effective} if its class lies in the
pseudo-effective cone.
\end{defi}

\begin{rmq}
    We denote by $N^1(X)$ the \emph{Néron–Severi group} of X defined by 
\[N^1(X) = \bigslant{\Div(X)}{\Num(X)}\]
of numerical equivalence classes of divisors on $X$.
\end{rmq}

Let us recall the following theorem (see \cite[Theorem~0.2]{boucksom_pseudo-effective_2012}), giving an equivalent definition for pseudo-effective line bundle.

\begin{thm}\label{pseff-boucksom}
    A line bundle $L$ on a projective manifold $X$ is pseudo-effective if, and only if, $L\cdot C \geq 0$ for all irreducible curves $C$ which move in a family covering $X$.
\end{thm}

This characterization allows us to prove the following result.

\begin{prop}\label{zero-dim-pseff}
    Let $X$ be a projective variety and $(Y,\Delta)$ be a projective orbifold pair. If $(Y,\Delta)$ is hyperbolic and $K_Y+\Delta$ is pseudo-effective, then $\Sur(X,(Y,\Delta))$, the set of surjective orbifold maps from $X$ on $(Y,\Delta)$, is zero-dimensional.
\end{prop}

\begin{rmq}
    This result implies \cref{zero-dim-big} and \cref{zero-dim-nef}.
\end{rmq}

\begin{proof}[Proof of \cref{zero-dim-pseff}]
We argue by contradiction.
Assume that there exists an irreducible component $\F$ of $\Sur(X,(Y,\Delta))$, such that $\dim \F\geq 1$.
Since we consider algebraic varieties, we can intersect $\F$ with sufficiently many hyperplanes, in such a way that $\dim \F= 1$.
By \cref{inters-neg-fort}, depending on the choice of the point $x\in X$, there exists a family of curves $C_x$, given by $C_x= ev_x(\F)$, which intersects negatively $(K_Y+\Delta)^{\otimes M}$, for an integer $M\in \N$.
Since $\F$ is a subset of surjective maps, one can move the point $x$, such that we obtain a family of curves covering $Y$. Indeed, for every point $y\in Y$, since $\mathcal{F}\subset \Sur(X,(Y,\Delta))$, there exist a point $x\in X$ and a map $f\in \F$ such that $y=f(x)$.

Then, by \cref{pseff-boucksom}, we obtain a direct contradiction with the fact that $K_Y+\Delta$ is pseudo-effective and thus we conclude the proof of \cref{zero-dim-pseff}.
\end{proof}

\begin{thm}\label{finite-pseff}
       Let $X$ be a projective variety and $(Y,\Delta)$ be a projective orbifold pair. If $(Y,\Delta)$ is hyperbolic and $K_Y+\Delta$ is pseudo-effective, then $\Sur(X,(Y,\Delta))$ is finite.
\end{thm}

The proof of this result is the same as the proof of \cref{finite-unif}, except that \cref{zero-dim-unif} is replaced by \cref{zero-dim-pseff}.

\bibliographystyle{alpha}
\bibliography{references.bib}

\begin{thebibliography}{BDPP12}

\bibitem[BDPP12]{boucksom_pseudo-effective_2012}
Sébastien Boucksom, Jean-Pierre Demailly, Mihai Păun, and Thomas Peternell.
\newblock The pseudo-effective cone of a compact {Kähler} manifold and
  varieties of negative {Kodaira} dimension.
\newblock {\em Journal of Algebraic Geometry}, 22(2):201--248, May 2012.

\bibitem[BJ24]{bartsch_kobayashi-ochiais_2024}
Finn Bartsch and Ariyan Javanpeykar.
\newblock Kobayashi-{Ochiai}'s finiteness theorem for orbifold pairs of general
  type.
\newblock {\em Journal of the Institute of Mathematics of Jussieu}, pages
  1--20, April 2024.

\bibitem[BJR23]{bartsch_weakly-special_2023}
Finn Bartsch, Ariyan Javanpeykar, and Erwan Rousseau.
\newblock Weakly-special threefolds and non-density of rational points,
  November 2023.
\newblock arXiv:2310.09065.

\bibitem[Cam04]{campana_orbifolds_2004}
Frédéric Campana.
\newblock Orbifolds, special varieties and classification theory.
\newblock {\em Annales de l’institut Fourier}, 54(3):499--630, 2004.

\bibitem[Cam05]{campana_fibres_2005}
Frédéric Campana.
\newblock Fibres multiples sur les surfaces: aspects geométriques,
  hyperboliques et arithmétiques.
\newblock {\em manuscripta mathematica}, 117(4):429--461, August 2005.

\bibitem[Cam11]{campana_orbifoldes_2011}
Frédéric Campana.
\newblock Orbifoldes géométriques spéciales et classification biméromorphe
  des variétés kählériennes compactes.
\newblock {\em Journal of the Institute of Mathematics of Jussieu},
  10(4):809--934, October 2011.

\bibitem[Cla15]{claudon_expose_2015}
Benoît Claudon.
\newblock {Positivité} du cotangent logarithmique et conjecture de
  {Shafarevich}-{Viehweg}, d'après {Campana}, {Păun}, {Taji}.
\newblock {\em Séminaire Bourbaki du 7 novembre 2015}, 2015.

\bibitem[CS08]{corlette_classification_2008}
Kevin Corlette and Carlos~T. Simpson.
\newblock On the classification of rank two representations of quasiprojective
  fundamental groups.
\newblock {\em Compositio Mathematica}, 144(5):1271--1331, September 2008.

\bibitem[CW09]{campana_brody_2009}
Frederic Campana and Jörg Winkelmann.
\newblock A {Brody} theorem for orbifolds.
\newblock {\em manuscripta mathematica}, 128(2):195--212, February 2009.

\bibitem[Del08]{delzant_trees_2008}
Thomas Delzant.
\newblock Trees, {Valuations} and the {Green}–{Lazarsfeld} {Set}.
\newblock {\em Geometric and Functional Analysis}, 18(4):1236--1250, December
  2008.

\bibitem[DF13]{de_franchis_teorema_1913}
Michele De~Franchis.
\newblock Un teorema sulle involuzioni irrazionali.
\newblock {\em Rendiconti del Circolo Matematico di Palermo}, page 368, July
  1913.

\bibitem[DR20]{darondeau_quasi-positive_2020}
Lionel Darondeau and Erwan Rousseau.
\newblock Quasi-positive orbifold cotangent bundles ; {Pushing} further an
  example by {Junjiro} {Noguchi}, August 2020.
\newblock arXiv:2006.13515.

\bibitem[Hol98]{holzapfel_ball_1998}
Rolf-Peter Holzapfel.
\newblock {\em Ball and {Surface} {Arithmetics}}, volume~29 of {\em Aspects of
  {Mathematics}}.
\newblock Vieweg+Teubner Verlag Wiesbaden, 1998.

\bibitem[Hor85]{horst_compact_1985}
Camilla Horst.
\newblock Compact varieties of surjective holomorphic endomorphisms.
\newblock {\em Mathematische Zeitschrift}, 190(4):499--504, December 1985.

\bibitem[KO75]{kobayashi_meromorphic_1975}
Shoshichi Kobayashi and Takushiro Ochiai.
\newblock Meromorphic mappings onto compact complex spaces of general type.
\newblock {\em Inventiones Mathematicae}, 31(1):7--16, February 1975.

\bibitem[Kob80]{kobayashi_first_1980}
Shoshichi Kobayashi.
\newblock The first {Chern} class and holomorphic symmetric tensor fields.
\newblock {\em Journal of the Mathematical Society of Japan}, 32(2), April
  1980.

\bibitem[Kob98]{kobayashi_hyperbolic_1998}
Shoshichi Kobayashi.
\newblock {\em Hyperbolic complex spaces}.
\newblock Springer, 1998.

\bibitem[Laz04]{lazarsfeld_positivity_2004-1}
Robert Lazarsfeld.
\newblock {\em Positivity in algebraic geometry. 1: {Classical} setting: line
  bundles and linear series}.
\newblock Springer, 2004.

\bibitem[MM86]{miyaoka_numerical_1986}
Yoichi Miyaoka and Shigefumi Mori.
\newblock A {Numerical} {Criterion} for {Uniruledness}.
\newblock {\em The Annals of Mathematics}, 124(1):65, July 1986.

\bibitem[Nog85]{noguchi_hyperbolic_1985}
Junjiro Noguchi.
\newblock Hyperbolic fibre spaces and {Mordell}'s conjecture over function
  fields.
\newblock {\em Publications of the Research Institute for Mathematical
  Sciences}, 21(1):27--46, 1985.

\bibitem[Nog86]{Noguchi_logarithmic_1986}
Junjiro Noguchi.
\newblock Logarithmic {Jet} {Spaces} and {Extensions} of de {Franchis}’
  {Theorem}.
\newblock In {\em Contributions to {Several} {Complex} {Variables}}, pages
  227--249. Vieweg+Teubner Verlag, Wiesbaden, 1986.

\bibitem[Nog92]{noguchi_meromorphic_1992}
Junjiro Noguchi.
\newblock Meromorphic mappings into compact hyperbolic complex spaces and
  geometric diophantine problems.
\newblock {\em International Journal of Mathematics}, 03(02):277--289, April
  1992.

\bibitem[Rou10]{rousseau_hyperbolicity_2010}
Erwan Rousseau.
\newblock Hyperbolicity of geometric orbifolds.
\newblock {\em Transactions of the American Mathematical Society},
  362(7):3799--3826, July 2010.

\bibitem[{Sta}18]{the_stacks_project_authors_stacks_2018}
The {Stacks Project Authors}.
\newblock {\em Stacks {Project}}.
\newblock https://stacks.math.columbia.edu/, 2018.

\end{thebibliography}

\end{document}